\documentclass[11pt,a4paper,reqno]{amsart}

\usepackage{amsfonts,amsthm,amscd,epsfig,amsmath,amssymb,enumerate }

\usepackage[all]{xy}

\numberwithin{equation}{section}

\DeclareMathSymbol{\leqslant}{\mathalpha}{AMSa}{"36} 
\DeclareMathSymbol{\geqslant}{\mathalpha}{AMSa}{"3E} 
\DeclareMathSymbol{\eset}{\mathalpha}{AMSb}{"3F}     
\renewcommand{\leq}{\;\leqslant\;}                   
\renewcommand{\geq}{\;\geqslant\;}                   

\newcommand{\be}{\begin{equation}}

\def\1{\ifmmode {1\hskip -3pt \rm{I}} \else {\hbox {$1\hskip -3pt \rm{I}$}}\fi}

\newtheorem{Th}{Theorem}[section]
\newtheorem{Le}[Th]{Lemma}

\newtheorem{Cor}[Th]{Corollary}
\newtheorem{Def}[Th]{Definition}
\newtheorem{Rem}{Remark}

\title[Cyclic random walks]
{Which random walks are cyclic? }

\author{D. Gabrielli}
 \address{Davide Gabrielli.  Dipartimento di Matematica, Universit\`a
 dell'Aquila, via Vetoio loc.  Coppito, 67100 L'Aquila, Italy. e--mail:
 gabriell@univaq.it, dvd.gabrielli@gmail.com}

\author{C. Valente}
 \address{Carla Valente.  Dipartimento di Matematica, Universit\`a
 dell'Aquila, via Vetoio  loc. Coppito, 67100 L'Aquila, Italy. e--mail:
 carla.valente@gmail.com}

\begin{document}

\begin{abstract}
A cyclic random walk is a random walk whose transition
probabilities/rates can be written as a superposition of the
empirical measures of a family of finite cycles. This identifies a
convex set of models. We discuss the problem of characterization of
cyclic random walks in some special cases showing that it is related
to several remarkable and classical results. In particular we
introduce the notion of balanced measure and show that a translation
invariant random walk on $\mathbb Z^d$ is cyclic if and only if its
transition probability is balanced. The characterization of the
extremal elements is obtained using the Carath\'{e}odory's Theorem
of convex analysis. We then show that a random walk on a finite set
is cyclic if and only if at every vertex the outgoing flux of the
transition graph is equal to the ingoing flux. The extremal elements
are characterized by the Birkhoff-Von-Neumann Theorem. Finally we
consider the discrete torus and discuss when the cyclic
decomposition can be  done using only homotopically trivial cycles
or elementary cycles associated to edges and two dimensional faces.
While in one dimension this is equivalent to require some geometric
properties of a discrete vector field associated to the transition
rates, in two dimension this is not the case. In particular we give
a simple characterization of the polyhedron of the rates admitting a
cyclic decomposition with elementary cycles. The proof is based on a
discrete Hodge decomposition, elementary homological algebra and the
Helly's Theorem of convex analysis. Finally we discuss a natural
discretization procedure of smooth divergence free continuous vector
fields and an application to random walks in random environments.

\medskip

\noindent {\sl Key words}: Random walks, convex analysis, discrete
geometry.

\medskip

\noindent {\sl AMS 2010 subject classification}:
82B41 
05C81 
52C99
\end{abstract}

\maketitle

\section{Introduction}

A discrete/continuous time cyclic random walk is a random walk whose
transition probabilities/rates can be written as a superposition of
empirical measures of a family of finite cycles. This definition
identifies a convex set of models. The aim of this paper is to
characterize this set in some special cases. This is important since
once you known that a random walk is cyclic then several properties
and results can be deduced. For example the following results have
been obtained  using some kind of cyclic decomposition: in
\cite{LKO} and \cite{KO} a central limit theorem and in \cite{D} a
quenched central limit theorem for random walks in random
environments, in \cite{S} some regularity properties of the
diffusion coefficient for a mean zero exclusion process, in \cite{V}
a central limit theorem for a tagged particle of a mean zero
exclusion process, in \cite{Math} some bounds on the transition
probabilities. In different papers some slightly different cyclic
decompositions are used. In any case all of them correspond to
define, starting from the Markov model, a weighted oriented graph
for which a cyclic decomposition holds. Our results can be
interpreted as results concerning cyclic decompositions of weighted
oriented graphs. In this sense they can be applied in any case
independently from the special rule associating a oriented weighted
graph to a Markov model. There is an huge amount of results
concerning cyclic decompositions of graphs, see for example
\cite{BJG}, \cite{B} and \cite{Di}. The majority of them are of
combinatorial type. Our results have instead a geometric flavour and
are the following.

\smallskip

We consider translational invariant Markov chains on the lattice
$\mathbb Z^d$. A model is identified by a probability measure on
$\mathbb Z^d$ which determines the weight to be associated to the
different jumps. We introduce the notion of balanced measure. An
Hyperplane on $\mathbb R^d$ containing the origin determines two
half spaces.  We can then compute the average distance from the
hyperplane for the measure restricted to the two half spaces. If the
two averaged distances are equal (possibly also $+\infty$) and if
this happens for any hyperplane then we say that the measure is
balanced. Clearly any mean zero measure is balanced. A translational
invariant Markov chain is cyclic if and only if the corresponding
measure determining the distribution of the jumps is balanced. A
characterization of the extremal elements is obtained using the
classic Carath\'{e}odory's Theorem of convex analysis whose
statement is recalled in section \ref{qui}. The result for mean zero
measures could be deduced by the results in \cite{K} and \cite{W}.
Here we give an independent proof and extend the validity of a
cyclic decomposition to the class of balanced measures. This result
is not a special case of the Choquet-Bishop-de Leeuw theorem
\cite{P} since compactness is missing. See \cite{WW} and \cite{WW2}
for a general discussion of problems of this type.

\smallskip

On a finite set a Markov model is determined by an oriented weighted
graph. An oriented weighted graph is said to be balanced if for any
vertex the ingoing weight is equal to the outgoing weight. We show
that a Markov model on a finite set is cyclic if and only if the
corresponding oriented weighted graph is balanced. In the case of
discrete time Markov chains, a characterization of the extremal
elements is obtained using the classic Birkhoff-Von Neumann Theorem
whose statement is recalled in section \ref{qui}. The proofs of
these results are elementary but are useful in the discussion of the
next issue. In \cite{BFG} we will discuss and apply a generalization of this
result valid for infinite graphs with a condition of vanishing flux
towards infinity.

\smallskip

On a finite graph we can consider a restricted class of cycles
taking into account some topological obstructions. More precisely we
can consider homotopically trivial cycles or elementary cycles
associated to one and two dimensional faces of some cellular
decomposition. We discuss this problem in the cases of the one
dimensional and the two dimensional discrete torii. While in one
dimension the validity of restricted cyclic decomposition of this
type is equivalent to require some geometric properties of a
discrete vector field constructed starting from the weights, this is
not the case in two dimensions. In particular we give a simple
characterization of the models that admit a cyclic decomposition
using elementary cycles. Similar computations can be done also in
higher dimensions. In these cases it is necessary a more detailed
discussion of the homological structure and a more involved
geometric construction. We will discuss this generalization in a
separated paper \cite{DG} together with a discussion of the case of
infinite grids..

\smallskip

We discuss also some applications. First we introduce a simple and
natural discretization procedure for a continuous smooth divergence
free vector field. The result is a discrete divergence free vector
field on the lattice. Then, using the previous results, we give a
condition to obtain a weighted oriented graph admitting an
elementary cyclic decomposition. Finally we apply the above
construction to deduce a quenched Central Limit Theorem for a random
walk in random environments using the results in \cite{D}.

\smallskip

The paper is organized as follows. In section \ref{Nmr} we fix
notation and state our main results. The section is subdivided into
four subsections according to the three different frameworks that
we discuss plus a subsection for applications. In section \ref{qui} we recall some classic results and
prove some original results that are useful in the proofs of our
main results. In section \ref{santorogo} we collect the proofs of
our main results in the case of a translational invariant Markov
chain on the lattice $\mathbb Z^d$. In section \ref{crwfg} we
collect the proofs of our main results in the case of a Markov model
on a finite set. In section \ref{crwfgt} we collect the proofs of
our main results concerning cyclic decompositions using restricted
classes of cycles. In section \ref{applications} we discuss the applications.

\section{Notation and main results}\label{Nmr}
In this section we state all the main results of the paper and fix
notation.
\subsection{Cyclic random walks on $\mathbb Z^d$}\label{ciclzd}
An element $x\in \mathbb Z^d$ has coordinates $(x_1,\dots,x_d)$. We
call $e^{(i)}$ the versors of the canonical basis of $\mathbb R^d$.
This means that we have $e^{(i)}_j:=\delta_{i,j}$, where $\delta$ is
the Kronecker delta. We denote by $x\cdot y:=\sum_{i=1}^dx_iy_i$ the
Euclidean scalar product among $x,y\in \mathbb R^d$. With
$|x|:=\sqrt{x\cdot x}$ we denote the Euclidean norm of $x$.

We denote by $\mathcal M^{\leq 1}$ the set of positive measures on
$\mathbb Z^d$ with total mass less or equal to $1$ endowed with the
topology of weak convergence. Since in all the cases we will
consider a tightness condition will be always satisfied, the weak
convergence will be equivalent to the pointwise convergence. In
$\mathcal M^{\leq 1}$ there is a natural partial order structure.
Given $p,q\in \mathcal M^{\leq 1}$ then $p \preceq q$ when $p(x)\leq
q(x)$ for any $x\in \mathbb Z^d$. The support of $p\in \mathcal
M^{\leq 1}$ is defined as
\begin{equation}
\mathcal S(p):=\left\{x\in \mathbb Z^d\,:\, p(x)>0\right\}\,.
\nonumber
\end{equation}
By $\mathcal M^1\subseteq \mathcal M^{\leq 1}$ we denote the subset
of probability measures and by $\mathcal M_0\subseteq \mathcal
M^{\leq 1}$ the subset of mean zero measures. This is the set of
measures $p$ such that
\begin{equation}
\sum_{x\in \mathbb Z^d}p(x)|x_i|<+\infty\,, \qquad \forall i\,,
\label{summas}
\end{equation}
and moreover $\sum_{x\in \mathbb Z^d}p(x)x=0$.

\smallskip

Given $A\subseteq \mathbb R^d$ we denote by $co(A)$ its convex hull
and by $aff(A)$ its affine hull. These are defined as
$$
co(A):=\left\{x\in \mathbb R^d\, :\, x=\sum_{i=1}^n\alpha_ix^i;\,
x^i\in A,\, \alpha_i\geq 0,\, \sum_{i=1}^n\alpha_i=1,\, n\in \mathbb
N\right\}\,,
$$
$$
aff(A):=\left\{x\in \mathbb R^d\, :\, x=\sum_{i=1}^n\alpha_ix^i;\,
x^i\in A,\, \alpha_i\in \mathbb R,\, \sum_{i=1}^n\alpha_i=1,\, n\in
\mathbb N\right\}\,.
$$

\smallskip

Let $z^0, z^1,\dots ,z^{n-1}$ be distinct elements of $\mathbb Z^d$
and consider also $z^n:=z^0$. Define
$$
C:=(z^0,z^1,\dots ,z^{n-1},z^n)
$$
as the finite cycle on $\mathbb Z^d$ that, starting from $z^0$,
visits sequentially all the elements $z^i$ in increasing order with
respect to the index $i$. We say that such a cycle $C$ has
cardinality $|C|=n$. Given a finite cycle $C$ we define its
empirical measure $p^C\in \mathcal M^1$ as
\begin{equation}
p^C:=\frac{1}{|C|}\sum_{i=1}^{|C|}\delta_{z^i-z^{i-1}}\,,
\label{defcicl}
\end{equation}
where $\delta_x$ is the delta measure concentrated at $x$. The
cycles of cardinality $1$ are of the type $C=(z^0,z^0)$ and the
corresponding empirical measure is $\delta_0$. Empirical measures of
finite cycles are called purely cyclic measures. It is clear that
$p^C$ defined in \eqref{defcicl} depends only on the values of the
displacement vectors $w^i:=z^i-z^{i-1}$ and not on their relative
order. Note also that for any cycle $C$ we have
\begin{equation}
\sum_{i=1}^{|C|}w^i=\sum_{i=1}^{|C|}(z^i-z^{i-1})=0\,. \nonumber
\end{equation}
As a consequence we have
$$
\sum_{x\in \mathbb Z^d}p^C(x)x=\frac{1}{|C|}\sum_{i=1}^{|C|}w^i=0\,,
$$
so that every purely cyclic measure has mean zero.

\smallskip

Consider a time homogeneous, discrete time random walk
$\left\{X_n\right\}_{n\in \mathbb N}$ on $\mathbb Z^d$ with
translation invariant transition probabilities determined by $p\in
\mathcal M^1$ by the following identification
\begin{equation}
\mathbb P(X_{n+1}= y|X_n=x):=p(y-x)\,. \label{def}
\end{equation}

\smallskip

A translation invariant discrete time random walk on $\mathbb Z^d$
is called  purely cyclic if the measure $p$ in \eqref{def} is purely
cyclic. This means that there exists a finite cycle $C$ such that
$p=p^C$.

\smallskip

As far as $p^C$ is concerned, we can naturally introduce an
equivalence relation $\sim$ between cycles. Let $C$ be a cycle and
let $\{w^1,\dots ,w^{|C|}\}$ be the corresponding set of
displacement vectors. Note that in the set of displacement vectors
of a cycle, a vector $w$ may appear more than once. For this reason
we will also denote the set of displacement vectors as
\begin{equation}
\{(w^1,n_1),\dots,(w^k,n_k)\}\,,\nonumber
\end{equation}
to indicate that the vector $w^i$ appears $n_i\in \mathbb N$ times.
Consider now another cycle $C'$ and let $\{v^1,\dots ,v^{|C'|}\}$ be
the corresponding set of displacement vectors.  We say that $C\sim
C'$ if
$$
\{w^1,\dots ,w^{|C|}\}=\{v^1,\dots ,v^{|C'|}\}\,.
$$
The equivalence class to which the cycle $C$ belongs is denoted by
$[C]$ and is determined by the set of the displacement vectors of
any representant $C$. We can then naturally use the following
identification
\begin{equation}
[C]\equiv\{(w^1,n_1),\dots,(w^k,n_k)\}\,. \label{arrLan}
\end{equation}

We denote by $\mathcal C$ the countable set of equivalence classes
of cycles, endowed with the discrete topology. Note that if
$\{(w^1,n_1),\dots,(w^k,n_k)\}\in \mathcal C$ then
\begin{equation}
\sum_{i=1}^kn_iw^i=0\,. \label{sum0}
\end{equation}
Conversely if \eqref{sum0} holds then there exist natural numbers
$n_i'$ such that $n_i'\leq n_i$ and moreover
$\{(w^1,n_1'),\dots,(w^k,n_k')\}\in \mathcal C$. We cannot use
directly the numbers $n_i$ since according to our definition a cycle
is self-avoiding. Observe that if $C\sim C'$ then $p^C=p^{C'}$ and
we can therefore use the notation $p^{[C]}$ to denote $p^C$ where
$C$ is any representant of $[C]$. The converse is not necessarily
true. Indeed let $l\in \mathbb N$, $l\geq 2$, and consider any cycle
$C$ satisfying \eqref{arrLan} and any cycle $C'$ with
\begin{equation}
[C']=\{(w^1,ln_1),\dots,(w^k,ln_k)\}\,.\nonumber
\end{equation}
Then clearly it holds $[C]\neq[C']$ but nevertheless $p^C=p^{C'}$.

\begin{Def}
\label{rimd}
A probability measure $p$ on $\mathbb Z^d$ is called
cyclic if there exists a probability measure $\rho$ on $\mathcal C$
such that $p=p^\rho$ where $p^\rho$ is defined as
\begin{equation}
p^\rho:=\sum_{[C]\in \mathcal C}\rho([C])p^{[C]}\,.\label{decomp}
\end{equation}
A translation invariant discrete time random walk on $\mathbb Z^d$
is called cyclic if the measure $p$ in \eqref{def} is cyclic.
\end{Def}

The meaning of \eqref{decomp} is that for any $x\in \mathbb Z^d$ we
have
\begin{equation}
p^\rho(x)=\sum_{[C]\in \mathcal C}\rho([C])p^{[C]}(x)\,,
\label{ben10}
\end{equation}
i.e. we require just pointwise convergence. By monotone convergence
Theorem \eqref{ben10} implies
\begin{equation}
p^\rho(A)=\sum_{[C]\in \mathcal C}\rho([C])p^{[C]}(A)\,, \qquad
\forall A\subseteq \mathbb Z^d\,. \nonumber
\end{equation}
Equivalently we have a probability measure on $\mathcal C\times
\mathbb Z^d$ that gives weight $\rho([C])p^{[C]}(z)$ to the pair
$([C],z)$ and $p^\rho$ is its $\mathbb Z^d$ marginal. Definition
\eqref{decomp} can also be interpreted as follows. For any sequence
$\mathcal C_n\subseteq \mathcal C$ such that $|\mathcal
C_n|<+\infty$, $\mathcal C_n\subseteq \mathcal C_{n+1}$ and
$\mathcal C=\cup_{n}\mathcal C_n$ it holds
\begin{equation}
p^\rho=\lim_{n\to +\infty}\sum_{[C]\in \mathcal
C_n}\rho([C])p^{[C]}\,, \label{igora}
\end{equation}
where the limit is in the weak sense. This follows immediately by
pointwise convergence since a tightness condition is clearly
verified.

If the measure $p^{\rho}$ is integrable, i.e. if \eqref{summas}
holds with $p^\rho$ instead of $p$, then by Fubini Theorem we get
$$
\sum_{x\in \mathbb Z^d}p^\rho(x)x=\sum_{x\in \mathbb
Z^d}\left(\sum_{[C]\in \mathcal
C}\rho([C])p^{[C]}(x)\right)x=\sum_{[C]\in\mathcal
C}\rho([C])\left(\sum_{x\in \mathbb Z^d}xp^{[C]}(x)\right)=0\,.
$$
This means that every measure $p^\rho$ that is integrable has mean
zero.

\medskip

Given $H$ an hyperplane in $\mathbb R^d$ we call $H^+$ and $H^-$ the
two closed half-spaces determined by $H$. Given $A,B\subseteq
\mathbb R^d$ we define
$$
d(A,B):=\inf_{\left\{x\in A\, ,\, y\in B\right\}}|x-y|\,.
$$
\begin{Def}\label{bal111}
A measure $p\in \mathcal M^{\leq 1}$ is called balanced if for
any hyperplane $H\ni 0$ we have
\begin{equation}
\sum_{x\in H^+}p(x)d(x,H)=\sum_{x\in H^-}p(x)d(x,H)\,.
\label{twosided}
\end{equation}
We call $\mathcal B\subseteq \mathcal M^1$ the set of balanced
probability measures.
\end{Def}

The main result of this subsection is the following.
\begin{Th}
\label{florisa}
An element of $\mathcal M^1$ is cyclic if and only
if it is balanced.
\end{Th}

We underline that in \eqref{twosided} both sides can also be
infinite. Note also that $\mathcal M_0\cap \mathcal M^1\subseteq
\mathcal B$. Indeed let $n$ be a versor normal to $H$ then we have
$H^+=\left\{x \, : \, x\cdot n\geq 0\right\}$ and $H^-=\left\{x \, :
\, x\cdot n\leq 0\right\}$. Given $x\in H^+$ it holds $d(x,H)=x\cdot
n$ and likewise given $x\in H^-$ it holds $d(x,H)=-x\cdot n$. If
$p\in \mathcal M_0\cap \mathcal M^1$ then we have
\begin{equation}
0=\sum_{x\in \mathbb Z^d}p(x)\left(x\cdot n\right)=\sum_{x\in
H^+\cap \mathbb Z^d}p(x)d(x,H)-\sum_{x\in H^-\cap \mathbb
Z^d}p(x)d(x,H)\,.\nonumber
\end{equation}
Indeed the same computation shows that when $p$ is integrable then
it is balanced if and only if it belongs to $\mathcal M_0\cap
\mathcal M^1$. This corresponds to verify condition \eqref{twosided}
just for the d hyper-planes orthogonal to the vectors $e^{(i)}$ of
the canonical basis. The validity of \eqref{twosided} for the others
hyperplanes will follow. When $p$ is not integrable the balancing
condition \eqref{twosided} has to be checked for any hyperplane $H$,
since the validity of \eqref{twosided} with respect to the
hyper-planes orthogonal to the vectors of the canonical basis does
not imply the validity of \eqref{twosided} for any $H$. Consider for
example the measure $p$ defined by
$p(2i,-i)=p(-i,2i)=\frac{3}{\pi^2i^2}\,, i\geq 1$ and giving weight
zero to all the remaining elements of $\mathbb Z^2$ (see figure
\ref{gugace}). This measure satisfies \eqref{twosided} with respect
to the 2 hyper-planes orthogonal to the vectors $e^{(1)}, e^{(2)}$
but it does not satisfy the balancing condition for example with
respect to the hyperplane $H$ in figure \ref{gugace}.
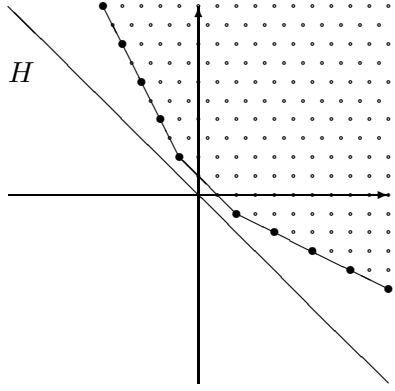
\begin{figure}
\setlength{\unitlength}{5cm}
\begin{picture}(1,1)
\put(0.5,0){\vector(0,1){1}} \put(0,0.5){\vector(1,0){1}}
\put(1,0){\line(-1,1){1}} \put(0,0.8){$H$} \put(0.45,
0.6){\circle*{0.02}} \put(0.4,0.7){\circle*{0.02}}
\put(0.35,0.8){\circle*{0.02}} \put(0.3,0.9){\circle*{0.02}}
\put(0.25,1){\circle*{0.02}} \put(0.6, 0.45){\circle*{0.02}}
\put(0.7,0.4){\circle*{0.02}} \put(0.8,0.35){\circle*{0.02}}
\put(0.9,0.3){\circle*{0.02}} \put(1,0.25){\circle*{0.02}}
\put(0.6,0.45){\line(-1,1){0.15}} \put(0.6,0.45){\line(2,-1){0.4}}
\put(0.45,0.6){\line(-1,2){0.2}}

\multiput(1,0.25)(0.05,0){1}%
{\circle{0.001}}
\multiput(0.9,0.3)(0.05,0){3}%
{\circle{0.001}}
\multiput(0.8,0.35)(0.05,0){5}%
{\circle{0.001}}
\multiput(0.7,0.4)(0.05,0){7}%
{\circle{0.001}}
\multiput(0.6,0.45)(0.05,0){9}%
{\circle{0.001}}

\multiput(0.55,0.5)(0.05,0){10}%
{\circle{0.001}}
\multiput(0.5,0.55)(0.05,0){11}%
{\circle{0.001}}

\multiput(0.45,0.6)(0.05,0){12}%
{\circle{0.001}}
\multiput(0.425,0.65)(0.05,0){12}%
{\circle{0.001}}
\multiput(0.4,0.7)(0.05,0){13}%
{\circle{0.001}}
\multiput(0.375,0.75)(0.05,0){13}%
{\circle{0.001}}
\multiput(0.35,0.8)(0.05,0){14}%
{\circle{0.001}}
\multiput(0.325,0.85)(0.05,0){14}%
{\circle{0.001}}
\multiput(0.3,0.9)(0.05,0){15}%
{\circle{0.001}}
\multiput(0.275,0.95)(0.05,0){15}%
{\circle{0.001}}
\multiput(0.25,1)(0.05,0){16}%
{\circle{0.001}}
\end{picture}
\caption{An example in $\mathbb Z^2$ of a measure $p$ that satisfies
condition \eqref{twosided} with respect to the 2 hyper-planes
orthogonal to $e^{(1)}$ and $e^{(2)}$ but it is not balanced and
consequently also not cyclic. With $\bullet$ we denote elements of
$\mathcal S(p)$, the dashed region represents $co(\mathcal S(p))$.
Condition \eqref{twosided} is not satisfied for the hyperplane $H$.}
\label{gugace}
\end{figure}
This measure is not cyclic. Indeed it is not possible to generate a
cycle using vectors in $\mathcal S(p)$ since $0\not \in co(\mathcal
S(p))$.

\smallskip

A simple example of a balanced not integrable probability measure on
$\mathbb Z^2$ is $p(x)=\frac{c}{|x|^{\frac 52}}$ where $c$ is a
suitable normalization constant. Another example is given by a
probability measure on $\mathbb Z^2$ such that
$$
\sum_{x_1\geq 0}p((x_1,0))|x_1|=\sum_{x_1\leq 0}p((x_1,0))|x_1|=\sum_{x_2\geq 0}p((0,x_2))|x_2|=\sum_{x_2\leq 0}p((0,x_2))|x_2|=+\infty\,.
$$
Then the above conditions guarantee that the measure is balanced and consequently cyclic, whatever are
the values assumed by the measure outside from the coordinate axis.

We discuss now in detail the one dimensional case. In one dimension
a balanced probability measure $p$ either belongs to $\mathcal
M_0\cap\mathcal M^1$ or satisfies
\begin{equation}
\sum_{x=1}^{+\infty}xp(x)=\sum_{x=1}^{+\infty}xp(-x)=+\infty\,.
\label{report}
\end{equation}
We now sketch an iterative procedure whose generalization is at the
core of the proof of Theorem \ref{florisa} and will be discussed
more in detail during its proof. This procedure generates a cyclic
decomposition of $p$ satisfying \eqref{report}. We remark that this case corresponds
to a not integrable balanced measure.  Given a measure
$p^l\in \mathcal M^{\leq 1}$ satisfying \eqref{report} we define
$$
\left\{
\begin{array}{l}
x^l_+:=\min\left\{x\geq 1\,:\, p^l(x)>0\right\}\,,\\
x^l_-:=\max\left\{x\leq -1\,:\, p^l(x)>0\right\}\,.
\end{array}
\right.
$$
We need to distinguish two cases
$$
\left\{
\begin{array}{ll}
p^l(x^l_+)x^l_++p^l(x^l_-)x^l_-
\geq 0\,, & \qquad(a)\,,\\
p^l(x^l_+)x^l_++p^l(x^l_-)x^l_- <0\,, & \qquad (b)\,.
\end{array}
\right.
$$
We consider the equivalence class of cycles
$$
[C^l]:=\left\{(x^l_+,n^l_+)\,,\, (x^l_-,n^l_-)\right\}\,,
$$
where $n^l_+$ and $n^l_-$ are determined requiring that
$\frac{n^l_+}{n^l_-}$ is an irreducible fraction and moreover
$\frac{n^l_+}{n^l_-}=-\frac{x^l_-}{x^l_+}$. The corresponding weight
will be
$$
\rho([C^l]):=\left\{
\begin{array}{ll}
\frac{p^l(x^l_-)(n^l_++n^l_-)}{n^l_-} & (a)\,,\\
\frac{p^l(x^l_+)(n^l_++n^l_-)}{n^l_+} & (b)\,.\\
\end{array}
\right.
$$
If we define $p^{l+1}:=p^l-\rho([C^l])p^{[C^l]}$ then $p^{l+1}\in
\mathcal M^{\leq 1}$ and still satisfies \eqref{report} so that we
can iterate the procedure starting from $p^1:=p-p(0)\delta_0$. By
construction we have $|\left\{\mathcal S(p^{l+1})\cap
[x_-^l,x_+^l]\right\}|\leq 1$ and condition \eqref{report}
guarantees that $\lim_{l\to +\infty}x^l_-=-\infty$ and $\lim_{l\to
+\infty}x^l_+=+\infty$. As a consequence we get
$$
p=p(0)\delta_0+\sum_{l=1}^{+\infty}\rho([C^l])p^{[C^l]}\,,
$$
that is a cyclic decomposition like \eqref{decomp}.

\smallskip

We call irreducible an equivalence class $[C]=\{w^1,\dots w^{|C|}\}$
such that for any non empty $\{v^1,\dots v^k\}\subset \{w^1,\dots
w^{|C|}\}$, where the inclusion is strict, it holds
$\sum_{i=1}^kv^i\neq 0$.

\smallskip

In lemma \ref{lemmaestrem}, stated and proved in section
\ref{santorogo}, we will characterize $\mathcal C^*$ a subset of the
irreducible equivalence classes of cycles. The second main result of
this section is the following.

\begin{Th}\label{teoesterm}
The set of extremal elements of $\mathcal B$ is
$\left\{p^{[C]}\right\}_{[C]\in \mathcal C^*}$. Moreover for any
element $p\in \mathcal B$ there exists a probability measure $\rho$
on $\mathcal C^*$ such that
\begin{equation}
p=\sum_{[C]\in \mathcal C^*}\rho([C])p^{[C]}\,. \label{airberlin}
\end{equation}
\end{Th}

We give also a different characterization of cyclic measures.  We
denote by $\overrightarrow{\mathcal M_0}$ the monotone
non-decreasing weak closure of $\mathcal M_0$. This means that $p
\in \overrightarrow{\mathcal M_0}$ if there exists a sequence
$p_n\in \mathcal M_0$ such that $p_n\succeq p_{n-1}$ and
$p=\lim_{n\to +\infty}p_n$ in the weak sense. Then we have the
following alternative characterization.
\begin{Th}
$$
\mathcal B=\overrightarrow{\mathcal M_0}\cap \mathcal M^1\,.
$$
\label{procbreve}
\end{Th}
The monotone requirement is necessary since it is easy to see that
$\overline{\mathcal M_0}\cap \mathcal M^1=\mathcal M^1$, where
$\overline{\mathcal M_0}$ is the weak closure of $\mathcal M_0$.

\smallskip

In this subsection we restricted our discussion to probability
measures but the arguments could be adapted to positive measures
having finite mass.

\subsection{Cyclic random walks on a finite set}
\label{lucilla} Both a discrete time and a continuous time Markov
chain on a finite or countable set $V$ can be described in terms of
a oriented weighted graph $(V,E,r)$. Here $V$ is the set of
vertices, $E\subseteq V\times V$ is the set of oriented edges and
$r:E\to \mathbb R^+$ is the weight function. For simplicity here and
hereafter we consider the case $(x,x)\not \in E$ for any $x\in V$,
the general case can be easily handled. The set of edges $E$ is
determined by the weight function by the requirement that $(x,y)\in
E$ if and only if $r(x,y)>0$. For this reason we can also denote
a weighted graph simply by the pair $(V,r)$. The set of edges having
positive weight is denoted by $E(r)$. There is a natural partial
order among weighted graphs having the same set of vertices. We say
that $(V,r)\preceq (V,r')$ if $r(x,y)\leq r'(x,y)$ for any $(x,y)\in
E(r)$. The set of all non negative weights $r: V\times V\to \mathbb R^+$
is denoted by $W$.

In the case of a discrete time random walk we fix the weights in
such a way that
$$
r(x,y):=\mathbb P(X_{n+1}=y|X_n=x)\,.
$$
In this case clearly the normalization condition
\begin{equation}
\sum_{\{y\in V : (x,y)\in E(r)\}}r(x,y)=1\,, \ \ \ \ \forall x\in
V\,,\label{norm}
\end{equation}
has to be satisfied. In the case of continuous time Markov chains we
identify the weight $r(x,y)$ with the rate of jump from $x$ to $y$.
\smallskip

A cycle $C$, on an oriented graph $(V,E)$, is a finite sequence
$C:=(x^0,x^1,\dots,x^{n-1},x^0)$ of distinct elements $x^i\in V$
such that for any $i=0,\dots,n-1$ we have $(x^i,x^{i+1})\in E$ where
the sum in the indices is modulo n. We say that an edge $(x,y)$
belongs to the cycle $C$ and write $(x,y)\in C$ if there exists an
$i$ such that $(x,y)=(x^i,x^{i+1})$. In this framework it is natural
to identify two cycles if they contain the same set of edges. More
precisely given two cycles $C=(x^0,x^1,\dots,x^{n-1},x^0)$ and
$C'=(y^0,y^1,\dots,y^{n-1},y^0)$, we say that $C\sim C'$ if there
exists a $j\in \{0,1,\dots,n-1\}$ such that $y^i=x^{i+j}$,
$i=0,1,\dots,n-1$, where the sum in the indices is modulo $n$. As
before we call $[C]$ the equivalence class of cycles having $C$ as a
representant. We call also $\mathcal C$ the set of equivalence
classes of cycles.

We associate to the cycle $C$ the weighted graph, that we call a
purely cyclic weighted graph, determined by the weight function
\begin{equation}
r^{C}(x,y):=\left\{
\begin{array}{ll}
1 & \hbox{if}  \ (x,y)\in C\,,\\
0 & \hbox{otherwise}\,.
\end{array}
\right. \label{pesociclo}
\end{equation}
Note that if $C\sim C'$ then $r^C=r^{C'}$ so that we can use the
notation $r^{[C]}$ without ambiguities. Moreover a purely cyclic
weighted graph uniquely determine an element of $\mathcal C$.

\begin{Def}
\label{cipolla}
We call a weighted graph $(V,r)$ cyclic if there
exists a positive measure $\rho$ on $\mathcal C$ such that
\begin{equation}
r=\sum_{[C]\in \mathcal C}\rho([C])r^{[C]}\,.\label{decgr}
\end{equation}
\end{Def}

\begin{Def}\label{bal222}
We will say that the weight $r$ is balanced at $x\in V$ if it
satisfies the condition
\begin{equation}
\sum_{\left\{y:(y,x)\in E(r)\right\}}r(y,x)=\sum_{\left\{y:(x,y)\in
E(r)\right\}}r(x,y)\,. \label{balancing}
\end{equation}
The weight $r$ is balanced if \eqref{balancing} holds at every $x\in
V$.
\end{Def}
We use the term balanced with two different meanings, one determined
by definition \ref{bal111} and the other one determined by
definition \ref{bal222}. There is no risk of confusion since one
refers to elements of $\mathcal M^{\leq 1}$ and the other one to
elements of $W$. The main result of this subsection is the
following.

\begin{Th}
A finite weighted graph $(V,r)$ is cyclic if and only if it is
balanced.\label{teofin}
\end{Th}

\begin{Rem}
Note that all the results of this paper can be formulated in terms
of cyclic decompositions of weighted graphs without any reference to
Markov models. A natural way to associate a weighted graph to a
Markov model is the one discussed at the beginning of this
subsection but it is clearly not the only one. For example given a
positive measure $\pi$ on $V$ and a random walk with jump rates
$c(x,y)$ we can associate to this pair a weighted graph with weights
\begin{equation}
r(x,y):=\pi(x)c(x,y)\,.\label{biblioroma}
\end{equation}
The definition of centered random walks in \cite{Math} is based for
example on the properties of a weighted graph defined in this way.
All the results of the present paper can be clearly applied also in
this case. Note that the validity of the balancing condition of Definition
\ref{bal222} for the rates in \eqref{biblioroma} is equivalent to
require that $\pi$ is invariant for the rates $c$.
\end{Rem}

\subsection{Cyclic random walks on a finite graph with topology}

We consider the $d$ dimensional continuous torus $\mathbb
T^d=\mathbb R^d/\mathbb Z^d$ of side length $1$. Any element of the
torus is an equivalence class that we identify with any of its
representant. On the continuous torus we embed a $d-$dimensional
discrete torus with mesh $\frac 1 N$. This is a graph whose set of
vertices is
\begin{equation}
V_N:=\left\{\mathbb T^d\ni x=\left(x_1,\dots,x_d\right)\, : \
x_1,\dots,x_d=0,\frac 1 N, \dots, \frac{N-1}{N}\right\}\,.
\end{equation}
The set of oriented edges $E_N$ contains all the pairs $(x,y)$, with
$x,y\in V_N$ and such that $d(x,y)\leq 1/N$, where the distance is
on $\mathbb T^d$. For simplicity we consider a cubic grid but all
the arguments can be repeated also in the case of a grid with
$N_1\times N_2\times \dots \times N_d$ sites.

We define $E'_N\subset E_N$ the collection of oriented edges
$(x,y)\in E_N$ such that there exist $\hat x, \hat y \in \mathbb
R^d$ in the same equivalence classes respectively of $x,y$ and such
that $\hat y=\hat x+e^{(i)}/N$ for some $i$. Note in particular that
for example $\left((\frac{N-1}{N}, x_2,\dots ,x_d),(0,x_2,\dots
,x_d)\right)\in E_N'$. This is the exact meaning that we give to
expressions like $(x,y)\in E_N$ with $y=x\pm e^{(i)}/N$.

The embedding of the discrete torus on the continuous one defines in
a natural way a cellular decomposition of $\mathbb T^d$. Every cell
of maximal dimension is a $d$ dimensional hyper-cube of side length
$\frac 1N$ having vertices in $V_N$. Every 2 dimensional cell of the
cellular decomposition is a square having vertices
$$
\left\{x,x+e^{(i)}/N,x+e^{(j)}/N,x+e^{(i)}/N+e^{(j)}/N\right\}\,,
$$
for some $x\in V_N$ and some $i\neq j$. Every 2 dimensional cell can
be oriented in 2 possible ways. An orientation is a choice among the
two possible elementary cycles
\begin{equation}
\left\{ \begin{array}{l}
\left(x,x+e^{(i)}/N,x+e^{(i)}/N+e^{(j)}/N,x+e^{(j)}/N,x\right)\,,\\
\left(x,x+e^{(j)}/N,x+e^{(i)}/N+e^{(j)}/N,x+e^{(i)}/N,x\right)\,,
\end{array}
\right. \label{occhiblu}
\end{equation}
associated to the 2 dimensional cell (see figure \ref{orcel} cases
(b) and (c)). We call $F_N$ the set of oriented two dimensional
cells of the cellular decomposition. Conventionally the two possible
orientations are called clockwise and anticlockwise. If $f\in F_N$
we call $f^c$ the element of $F_N$ associated to the same 2
dimensional cell but corresponding to the opposite orientation. We
call $F'_N\subset F_N$ the subset of anticlockwise oriented 2
dimensional cells, this means the cells with associated the
orientation
$$
\left(x,x+e^{(\min
\left\{i,j\right\})}/N,x+e^{(i)}/N+e^{(j)}/N,x+e^{(\max\left\{i,j\right\})}/N,x\right)\,.
$$

The discrete torus is a finite graph so that cycles and equivalence
classes of cycles are defined as in subsection \ref{lucilla}. In
particular we still call $\mathcal C$ the set of equivalence classes
of cycles. Given a cycle $C=(z^0,z^1,\dots, z^{n-1},z^0)$ on the
discrete torus we associate to it the displacement vectors $w^1, \dots w^n$ defined by
\begin{equation}
w^i:=\left\{
\begin{array}{ll}
\frac{e^{(j)}}{N} & \textrm{if}\ z^i=z^{i-1}+\frac{e^{(j)}}{N}\,,\\
-\frac{e^{(j)}}{N} & \textrm{if}\ z^i=z^{i-1}-\frac{e^{(j)}}{N}\,.
\end{array}
\right.
\end{equation}
\begin{figure}
\setlength{\unitlength}{5cm}
\begin{displaymath}
\put(0,-0.35){$(a)$} \put(0.4,-0.35){$(b)$}
\put(0.9,-0.35){$(c)$}\xymatrix{ \bullet \ar@/^/[d] &\bullet \ar[r]&
\bullet \ar[d] &
\bullet \ar[d]& \bullet \ar[l]\\
\bullet \ar@/^/[u] & \bullet \ar[u]& \bullet \ar[l]& \bullet \ar[r]
& \bullet \ar[u]
 }
\end{displaymath}
\caption{Elementary cycles associated to one dimensional cells $(a)$
and two dimensional cells $(b)$ and $(c)$, respectively clockwise
and anticlockwise oriented.} \label{orcel}
\end{figure}
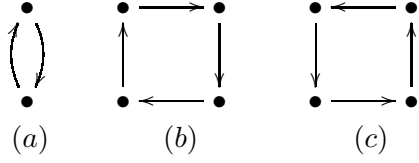
Given a cycle $C=(z^0,z^1,\dots, z^{n-1},z^0)$ on the discrete torus
we can naturally associate to it also a continuous closed curve on
$\mathbb T^d$. This is the projection on $\mathbb T^d$ of a
continuous curve on $\mathbb R^d$ whose parametrization $\left\{\hat
x(s)\right\}_{s\in [0,1]}$ is defined by the positions $\hat
x(0)=z^0$ and
$$
\dot{\hat x}(s)=nw^i\,, \qquad s\in\left[\frac{i-1}{n},
\frac{i}{n}\right)\,.
$$
Since the projection is a closed curve on $\mathbb T^d$ then
$\hat x(1)-\hat x(0)$ is a vector with integer coefficients. A
continuous closed curve on $\mathbb T^d$ is called homotopically
trivial if it can be continuously deformed to a single point. More
precisely a continuous closed curve
$\left\{x(s)\right\}_{s\in[0,1]}$ is homotopically trivial if there
exists a continuous map $X:[0,1]\times[0,1]\to \mathbb T^d$
such that $X(s,0)=x(s)$, $X(s,1)=x^*$ for any $s\in
[0,1]$, where $x^*\in \mathbb T^d$ is a fixed element and
$X(1,t)=X(0,t)$ for any
$t$. The discrete cycle $C$ will be called homotopically trivial if
the associated continuous curve on the torus is homotopically trivial. It is easy to see that this
condition is equivalent to require that
\begin{equation}
\sum_{i=1}^nw^i=0\,, \label{rivestimento}
\end{equation}
where $w^i$ are the displacement vectors of the cycle. By $\mathcal
C^* \subseteq \mathcal C$ we denote the subset of equivalence
classes associated to homotopically trivial cycles.

\smallskip

We define also a subset $\mathcal C^e\subseteq \mathcal C^*$ of
equivalence classes of elementary cycles. Elementary cycles are
naturally associated to one and two dimensional cells of the
cellular decomposition. Given $x,y\in V_N$ such that $d(x,y)=1/N$ we
call $[C_{\{x,y\}}]$ the equivalence class containing the elementary
cycle $C_{\{x,y\}}:=(x,y,x)$ (see Figure \ref{orcel} case (a)).

As already discussed to every element $f\in F_N$ corresponds a cycle
of the type \eqref{occhiblu}. We call $[C_f]$ the corresponding
equivalence class. The set of equivalence classes of elementary
cycles $\mathcal C^e$ is defined as
$$
\mathcal C^e:=\left(\left\{[C_{\{x,y\}}]\right\}_{(x,y)\in
E_N'}\right)\cup \left(\left\{[C_{f}]\right\}_{f\in F_N}\right)\,.
$$

Given $(x,y)\in E_N$ and $f\in F_N$ we write $(x,y)\in f$ if the
oriented edge $(x,y)$ belongs to any representant of $[C_f]$. A pair
of elements of $F_N$ associated to adjacent cells is said to be
oriented in agreement if no elements of $E_N$ belong to both. On the
two dimensional torus every pair of elements of $F_N'$ is oriented
in agreement (see figure \ref{fig3}). This is possible since the two
dimensional torus is an orientable surface.
\begin{figure}
\setlength{\unitlength}{5cm}
\begin{picture}(1,1)

\multiput(0,0)(0.2,0){5}%
{\line(0,1){0.2}}
\multiput(0,0.2)(0.2,0){5}%
{\line(0,1){0.2}}
\multiput(0,0.4)(0.2,0){5}%
{\line(0,1){0.2}}
\multiput(0,0.6)(0.2,0){5}%
{\line(0,1){0.2}}
\multiput(0,0.8)(0.2,0){5}%
{\line(0,1){0.2}}
\multiput(0,0)(0.2,0){5}%
{\line(1,0){0.2}}
\multiput(0,0.2)(0.2,0){5}%
{\line(1,0){0.2}}
\multiput(0,0.4)(0.2,0){5}%
{\line(1,0){0.2}}
\multiput(0,0.6)(0.2,0){5}%
{\line(1,0){0.2}}
\multiput(0,0.8)(0.2,0){5}%
{\line(1,0){0.2}}
\multiput(0,1)(0.2,0){5}%
{\line(1,0){0.2}}
\multiput(1,0)(0,0.2){5}%
{\line(0,1){0.2}}

\put(0.1,0.1){\circle{0.14}}\put(0.17,0.11){\vector(0,1){0.00001}}
\put(0.3,0.3){\circle{0.14}}\put(0.37,0.31){\vector(0,1){0.00001}}
\put(0.5,0.3){\circle{0.14}}\put(0.57,0.31){\vector(0,1){0.00001}}
\put(0.3,0.1){\circle{0.14}}\put(0.37,0.11){\vector(0,1){0.00001}}
\put(0.5,0.1){\circle{0.14}}\put(0.57,0.11){\vector(0,1){0.00001}}
\put(0.7,0.1){\circle{0.14}}\put(0.77,0.11){\vector(0,1){0.00001}}
\put(0.9,0.1){\circle{0.14}}\put(0.97,0.11){\vector(0,1){0.00001}}
\put(0.1,0.3){\circle{0.14}}\put(0.17,0.31){\vector(0,1){0.00001}}
\put(0.7,0.3){\circle{0.14}}\put(0.77,0.31){\vector(0,1){0.00001}}
\put(0.9,0.3){\circle{0.14}}\put(0.97,0.31){\vector(0,1){0.00001}}

\put(0.1,0.5){\circle{0.14}}\put(0.17,0.51){\vector(0,1){0.00001}}
\put(0.7,0.5){\circle{0.14}}\put(0.77,0.51){\vector(0,1){0.00001}}
\put(0.9,0.5){\circle{0.14}}\put(0.97,0.51){\vector(0,1){0.00001}}
\put(0.3,0.5){\circle{0.14}}\put(0.37,0.51){\vector(0,1){0.00001}}
\put(0.5,0.5){\circle{0.14}}\put(0.57,0.51){\vector(0,1){0.00001}}

\put(0.1,0.7){\circle{0.14}}\put(0.17,0.71){\vector(0,1){0.00001}}
\put(0.7,0.7){\circle{0.14}}\put(0.77,0.71){\vector(0,1){0.00001}}
\put(0.9,0.7){\circle{0.14}}\put(0.97,0.71){\vector(0,1){0.00001}}
\put(0.3,0.7){\circle{0.14}}\put(0.37,0.71){\vector(0,1){0.00001}}
\put(0.5,0.7){\circle{0.14}}\put(0.57,0.71){\vector(0,1){0.00001}}

\put(0.1,0.9){\circle{0.14}}\put(0.17,0.91){\vector(0,1){0.00001}}
\put(0.7,0.9){\circle{0.14}}\put(0.77,0.91){\vector(0,1){0.00001}}
\put(0.9,0.9){\circle{0.14}}\put(0.97,0.91){\vector(0,1){0.00001}}
\put(0.3,0.9){\circle{0.14}}\put(0.37,0.91){\vector(0,1){0.00001}}
\put(0.5,0.9){\circle{0.14}}\put(0.57,0.91){\vector(0,1){0.00001}}

\end{picture}
\caption{Elements of $F_N'$ of the two dimensional torus. Any pair
is oriented in agreement. Opposite sides of the square are identified. } \label{fig3}
\end{figure}
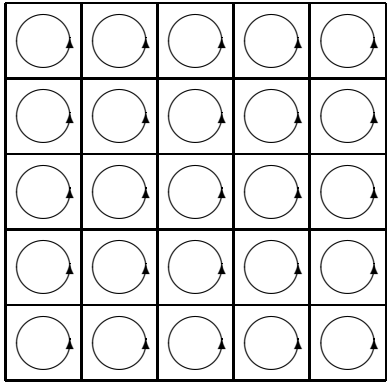

\smallskip
Given a continuous time random walk on the discrete torus its
transition rates identify a weight $r\in W$. The discrete torus is a
finite graph and a decomposition like \eqref{decgr} exists if and
only if the balancing condition of Definition \ref{bal222} holds. We investigate
under which conditions decompositions like
\begin{equation}
r=\sum_{[C]\in \mathcal C^*}\rho([C])r^{[C]}\,, \label{dec*}
\end{equation}
or
\begin{equation}
r=\sum_{[C]\in \mathcal C^e}\rho([C])r^{[C]}\,, \label{decel}
\end{equation}
hold. Recall that $\rho$ is a positive measure and $r^{[C]}$ is the
element of $W$ associated to any cycle in $[C]$ like in
\eqref{pesociclo}. Let us call $R^*\subseteq W$ and $R^e\subseteq W$
the sets of weights for which respectively a decomposition like
\eqref{dec*} and like \eqref{decel} holds. Clearly it holds
$R^e\subseteq R^*$.

\smallskip
Before proceeding we need to introduce some notions and terminology.
We refer to \cite{DKT}, \cite{daf}, \cite{Me} for more details. We
call $\Lambda^0$ the $|V_N|-$dimensional vector space of the real
functions $f: V_N\to \mathbb R$.

We call a function $\phi:E_N\to \mathbb R$ a discrete vector field
if it satisfies the condition
\begin{equation}
\phi(x,y)=-\phi(y,x)\,, \qquad \forall (x,y)\in E_N\,.\label{defdvf}
\end{equation}
We call $\Lambda^1$ the $|E'_N|-$dimensional vector space of
discrete vector fields. We endow $\Lambda^1$ of the scalar product $\langle
\cdot,\cdot \rangle_{1}$ defined as
\begin{equation}
\langle \phi,\phi'\rangle_{1}:=\sum_{(x,y)\in
E'_N}\phi(x,y)\phi'(x,y)\,. \label{scpr}
\end{equation}
Due to \eqref{defdvf} the value of \eqref{scpr} does not depend on
the specific choice of $E'_N$. On $W$ we define a projection operator that
associates to a weight $r$ a discrete vector field $\phi^r\in
\Lambda^1$ defined as
\begin{equation}
\phi^r(x,y):=r(x,y)-r(y,x).\nonumber
\end{equation}
Conversely to any $\phi\in\Lambda^1$ we canonically associate a weight
$r^\phi\in W$ defined as
\begin{equation}
r^\phi(x,y):=[\phi(x,y)]_+\,, \nonumber
\end{equation}
where $[\cdot]_+$ denotes the positive part of a real number
\begin{equation}
[a]_+:=\left\{
\begin{array}{ll}
a & \hbox{if}\ a\geq 0\,,\\
0 & \hbox{otherwise}\,.
\end{array}
\right.\nonumber
\end{equation}
Note that $\min \left\{r^{\phi}(x,y),r^{\phi}(y,x)\right\}=0$ and
$r^\phi\in R(\phi)$, where
$$
R(\phi):=\left\{r\in W\,:\,\phi^r=\phi\right\}\,.
$$
We stress that $R(0)$ coincides with the set of symmetric rates,
i.e. such that $r(x,y)=r(y,x)$. Given a positive weight $r$ we can
decompose it into the sum of two positive weights as
\begin{equation}
r=s+r^{\phi^r}. \label{nasce}
\end{equation}
Clearly $s\in R(0)$ and it is called the symmetric part of $r$. We
can deduce
\begin{equation}
R(\phi)=\left\{r\,:\, r=s+r^\phi\,,\ s\in
R(0)\right\}=r^\phi+R(0)\,. \label{baku}
\end{equation}
A simple consequence of \eqref{baku} is that there exists a minimal
element, with respect to the partial order $\preceq$, on $R(\phi)$
that is $r^\phi$.

We call a two chain a map $\psi:F_N\to \mathbb R$ that satisfies the
additional property
$$
\psi(f)=-\psi(f^c)\,, \qquad \forall f\in F_N\,.
$$
We call $\Lambda^2$ the $|F_N'|-$dimensional vector space of the
2-chains.

\smallskip

We define the co-boundary operator $\delta:\Lambda^i\to
\Lambda^{i+1}$, $i=0,1$ as follows. Given  a function $f\in
\Lambda^0$ we define $\delta f \in \Lambda^1$ as
$$
\delta f(x,y):=f(y)-f(x)\,, \qquad (x,y)\in E_N\,.
$$
A discrete vector field obtained in this way is called a gradient.
Given $\phi \in \Lambda^1$ we define $\delta \phi\in \Lambda^2$ as
$$
\delta \phi(f):=\sum_{\left\{(x,y)\in E_N\,:\,(x,y)\in
f\right\}}\phi(x,y)\,, \qquad f\in F_N\,.
$$
The value $\delta \phi(f)$ is called the value of the circulation of
the vector field $\phi$ around $f$.

We define also a boundary operator $d:\Lambda^i\to \Lambda^{i-1}$,
$i=1,2$. Given $\psi\in \Lambda^2$ we define $d\psi \in \Lambda^1$
as
\begin{equation}
d\psi(x,y):=\sum_{\left\{f\in F_N\,:\, (x,y)\in f\right\}}\psi(f)\,,
\ \ \ \ (x,y)\in E_N\,. \label{defdd}
\end{equation}
Given $\phi\in \Lambda^1$ we define $d\phi\in \Lambda^0$ as
\begin{equation}
d\phi(x):=\sum_{\left\{y\in V_N\,:\, (x,y)\in
E_N\right\}}\phi(x,y)\,, \qquad x\in V_N\,. \label{ddber}
\end{equation}
The r.h.s of \eqref{ddber} is called the discrete divergence at $x$
of $\phi$.

It is easy to verify according to our definitions that for any
$\psi\in \Lambda^2$ it holds
\begin{equation}
d(d\psi)=0\,.\label{dd=0}
\end{equation}

\smallskip

The one dimensional case is elementary and will be considered
separately. Here we restrict to dimensions $d\geq 2$. We denote by
$d\Lambda^2\subseteq \Lambda^1$ the set of vector fields $\phi$ for
which there exists a $\psi\in \Lambda^2$ with $\phi=d\psi$. A
characterization of $d\Lambda^2$ will be discussed in terms of a
discrete Hodge decomposition in section \ref{qui}. It can be
summarized as follows. A discrete vector field $\phi$ belongs to
$d\Lambda^2$ if and only if $d \phi(x)=0$ for any $x$ and
$$
\sum_{x\in V_N}\phi\left(x,x+e^{(i)}/N\right)=0\,, \qquad i=1,\dots
,d\,.
$$
We have the following result.
\begin{Le}
A necessary condition to have $r\in R^*$ is that $\phi^r\in
d\Lambda^2$. The same holds for $R^e$. \label{neclass}
\end{Le}

We have also the following Lemma where the metrics in $W$ is the
Euclidean one.
\begin{Le}\label{closed}
The sets $R^e$ and $R^*$ are closed sets in $W$.
\end{Le}
As we will see, an immediate consequence of the above Lemma is the following Corollary. It says that $\phi^r\in
d\Lambda^2$ is not a sufficient condition to have $r\in R^*$ (and
consequently also $r\in R^e$). We add a condition of positivity
since in this way we get a stronger statement.
\begin{Cor}
\label{corclosed} There exists a weight $r\in W$ such that
$\phi^r\in d\Lambda^2$, $r(x,y)>0$ for any $(x,y)\in E_N$ but
nevertheless $r\not \in R^*$.
\end{Cor}

\subsubsection{One dimensional torus}
In the one dimensional case there are no two dimensional faces and
$$
\mathcal C^*=\mathcal C^e=\left\{[C_{(x,x+1/N)}]\right\}_{x\in V_N}\,,
$$
so that $|\mathcal C^*|=N$. Moreover we have $|\mathcal C|=N+2$. To
the previous equivalence classes we need to add the two equivalence
classes having as representants the cycles
$$C_+:=\left(0,1/N,2/N,\dots,(N-1)/N,0\right)$$
and
$$C_-:=\left(0,(N-1)/N,(N-2)/N,\dots,1/N,0\right)\,.$$
The following Theorem says that the validity of decompositions of
the type \eqref{dec*} and \eqref{decel} is equivalent to a geometric
property of the vector field $\phi^r$. More precisely \eqref{dec*}
and \eqref{decel} hold if and only if $\phi^r$ is zero.
\begin{Th}
It holds a decomposition of the type \eqref{dec*} if and only if
$\phi^r=0$. Moreover this happens if and only if the rates $r$ are
reversible with respect to the uniform measure. \label{teodoro}
\end{Th}

Let
\begin{equation}
m:=\min_{(x,y)\in E_N}r(x,y)\,.
\nonumber
\end{equation}
The following Theorem says that also the validity of a decomposition
like \eqref{decgr} is still equivalent to a geometric property of
the vector field $\phi^r$. More precisely \eqref{decgr} holds if and
only if $\phi^r$ is divergence free. Moreover the model is so simple
that we can characterize completely the associated measures $\rho$.

\begin{Th}
\label{teodoro2} A decomposition like \eqref{decgr} holds if and
only if the discrete vector field $\phi^r$ is constant or
equivalently has zero divergence $d\phi^r=0$. Moreover if $\phi^r=c$
then all the measures on $\mathcal C$ for which \eqref{decgr} holds
are parameterized by the real parameter $a\in [0,m]$ as
\begin{equation}
\left\{
\begin{array}{l}
\rho([C_{\left\{x,y\right\}}])=\min\left\{r(x,y),r(y,x)\right\}-a\,,\\
\rho([C_+])=[c]_++a\,,\\
\rho([C_-])=[-c]_++a\,.
\end{array}
\right.\label{teonuovo}
\end{equation}
\end{Th}
By a constant discrete vector field $\phi=c$ we mean $\phi(x,y)=c$
for any $(x,y)\in E'_N$.
\subsubsection{Two dimensional torus}
Every $(x,y)\in E_N'$ belongs to only two elements of $F_N$ and
moreover one is clockwise oriented and the other one is
anticlockwise oriented. We call $f_+$ the element of $F_N'$ such
that $(x,y)\in f_+$ and $f_-$ the element of $F_N'$ such that
$(y,x)\in f_-$. Let $\phi \in d\Lambda^2$ and consider any $\psi\in
\Lambda^2$ such that $d\psi=\phi$. As will be explained more in
detail in section \ref{qui}, any other $\psi'$ such that
$d\psi'=\phi$ differs from $\psi$ by an additive constant, this
means that there exists a real constant $c$ such that
$$
\psi'(f)=\psi(f)+c\,,\qquad \forall f\in F_N'\,.
$$
The main result of this subsection, Theorem \ref{ugu} below, will
not depend on the specific choice of $c$. To every $(x,y)\in
E_N'$ we associate the closed interval $I(x,y)$ of $\mathbb R$
defined as
$$
I(x,y)=[i_1(x,y),i_2(x,y)]:=\left[\min\left\{\psi(f_-),\psi(f_+)\right\},
\max\left\{\psi(f_-),\psi(f_+)\right\}\right]\,.
$$

Given $I_1,\dots ,I_n$ a collection of closed intervals of $\mathbb
R$ we call $\mathcal P(I_1,\dots,I_n)\subseteq \mathbb R^n$ the
closed unbounded convex polyhedron defined as follows. An element
$s=(s_1,\dots,s_n)\in \mathbb R^n$ belongs to $\mathcal
P(I_1,\dots,I_n)$ if and only if the inequalities
\begin{equation}
s_i+s_j\geq d(I_i,I_j)\,, \ \ \ \ i,j=1,\dots,n\,, \label{vincpoli}
\end{equation}
are satisfied. Note that in \eqref{vincpoli} we are considering also
the cases $i=j$ that imply $s_i\geq 0$. See Figure \ref{tuttotace}
for a simple example in dimension $n=2$. Two important properties of
this polyhedron are the following. The first one is that if $s \in
\mathcal P(I_1,\dots,I_n)$ and $s'\geq s$ then $s'\in \mathcal
P(I_1,\dots,I_n)$. The second one is that $0\in \mathcal
P(I_1,\dots,I_n)$ if and only if $d(I_i,I_j)=0$ for any $i,j$. By
the Helly's Theorem recalled in section \ref{qui} this is equivalent
to $\cap_iI_i\neq \emptyset$. In this case clearly $\mathcal
P(I_1,\dots,I_n)=(\mathbb R^+)^n$.
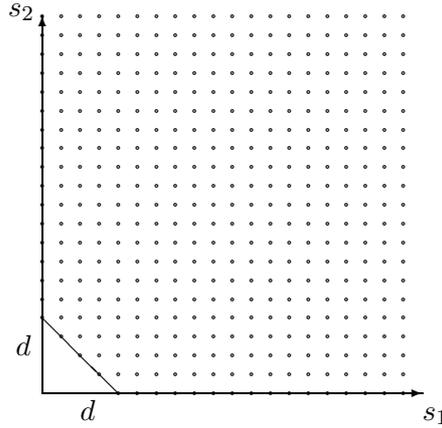
\begin{figure}
\setlength{\unitlength}{5cm}
\begin{picture}(1,1)
\put(1,-0.07){$s_1$} \put(-0.09,1){$s_2$} \put(0.1,-0.07){$d$}
\put(-0.07,0.1){$d$} \put(0,0){\vector(0,1){1}}
\put(0,0){\vector(1,0){1}} \put(0,0.2){\line(1,-1){0.2}}
\multiput(0,0.2)(0.05,0){20}%
{\circle{0.001}}
\multiput(0,0.25)(0.05,0){20}%
{\circle{0.001}}
\multiput(0,0.3)(0.05,0){20}%
{\circle{0.001}}
\multiput(0,0.35)(0.05,0){20}%
{\circle{0.001}}
\multiput(0,0.4)(0.05,0){20}%
{\circle{0.001}}
\multiput(0,0.45)(0.05,0){20}%
{\circle{0.001}}
\multiput(0,0.5)(0.05,0){20}%
{\circle{0.001}}
\multiput(0,0.55)(0.05,0){20}%
{\circle{0.001}}
\multiput(0,0.6)(0.05,0){20}%
{\circle{0.001}}
\multiput(0,0.65)(0.05,0){20}%
{\circle{0.001}}
\multiput(0,0.7)(0.05,0){20}%
{\circle{0.001}}
\multiput(0,0.75)(0.05,0){20}%
{\circle{0.001}}
\multiput(0,0.8)(0.05,0){20}%
{\circle{0.001}}
\multiput(0,0.85)(0.05,0){20}%
{\circle{0.001}}
\multiput(0,0.9)(0.05,0){20}%
{\circle{0.001}}
\multiput(0,0.95)(0.05,0){20}%
{\circle{0.001}}
\multiput(0,1)(0.05,0){20}%
{\circle{0.001}}
\multiput(0.05,0.15)(0.05,0){19}%
{\circle{0.001}}
\multiput(0.1,0.1)(0.05,0){18}%
{\circle{0.001}}
\multiput(0.15,0.05)(0.05,0){17}%
{\circle{0.001}}
\multiput(0.2,0)(0.05,0){16}%
{\circle{0.001}}

\end{picture}
\caption{The dashed region represents the polyhedron $\mathcal
P\left(I_1,I_2\right)\subseteq \left(\mathbb R^+\right)^2$ when
$d(I_1,I_2)=d>0$. In this case $0\not\in \mathcal
P\left(I_1,I_2\right)$.} \label{tuttotace}
\end{figure}

The main result of this subsection is the following. It says that,
differently from the one dimensional case, the validity of a
decomposition like \eqref{decel} is not simply equivalent to require a
geometric property of the vector field $\phi^r$, but involves
instead also the symmetric part $s$.

\begin{Th}
\label{ugu} Let $r$ such that $\phi^r=\phi\in d\Lambda^2$. We have
$r\in R^e$ if and only if
\begin{equation}
s\in \mathcal P\left(\left\{I(x,y)\right\}_{(x,y)\in E'_N}\right)\,.
\label{isolafa}
\end{equation}
\end{Th}
In \eqref{isolafa} $s$ is the symmetric part of $r$ (see
\eqref{nasce}) and the intervals are constructed using any $\psi\in
\Lambda^2$ such that $d\psi=\phi$. The polyhedron obtained is
independent on the specific choice.

\subsubsection{Other topologies}

We can generalize Theorem \ref{ugu} to surfaces different from the
two dimensional torus and to cellular decompositions different from
the cubic one. More precisely consider an unoriented graph embedded
on a compact surface without boundary. Vertices are associated to
points of the surface, edges are associated to continuous self
avoiding curves on the surface connecting vertices. Two different
curves may intersect only on vertices. Cutting the surface along
edges we obtain a finite number of two dimensional cells
homeomorphic to two dimensional balls. Every edge belong to only two
different two dimensional cells. From this we construct an oriented
graph $(V,E)$ whose vertices coincide with the vertices of the
unoriented graph and whose oriented edges are obtained splitting any
unoriented edge $\left\{x,y\right\}$ into $(x,y)$ and $(y,x)$.
Elementary cycles can be defined also in this case. They are
naturally associated to one dimensional cells and oriented two
dimensional cells and are defined like in the previous subsection.
Also the boundary and co-boundary operators are defined in a similar
way. We avoid formal definitions, see \cite{DKT}, \cite{daf} and
\cite{Me} for more details. We need to distinguish two cases, when
the surface is orientable or not. If the surface is orientable and
$\phi\in d\Lambda^2$ then $\psi\in \Lambda^2$ such that $\phi=d\psi$
is defined up to an additive constant. If the surface is non
orientable and $\phi\in d\Lambda^2$ then $\psi\in \Lambda^2$ such
that $\phi=d\psi$ is uniquely determined.

In the orientable case we can fix $F_N'\subseteq F_N$ choosing an
orientation for any two dimensional cell in such a way that any pair
of adjacent elements of $F_N'$ is oriented in agreement. We fix also
$E_N'\subseteq E_N$ choosing arbitrarily one among the two possible
orientations for any edge. For any $(x,y)\in E_N'$ there exists only
one element $f_+\in F_N'$ such that $(x,y)\in f_+$ and one element
$f_-\in F_N'$ such that $(y,x)\in f_-$. The corresponding interval
$I(x,y)$ is defined as in the case of the two dimensional torus. We
then have the following Theorem.

\begin{Th}
\label{bart} Consider a weighted oriented graph constructed starting
from a finite cellular subdivision of a compact orientable surface
without boundary and let $r\in W$ such that $\phi^r=\phi\in
d\Lambda^2$. Then we have $r\in R^e$ if and only if
$$
s\in \mathcal P\left(\left\{I(x,y)\right\}_{(x,y)\in E'_N}\right)\,.
$$
\end{Th}
Given $J_1,\dots ,J_n$ a collection of closed subsets of $\mathbb R$
we call $\mathcal P'(J_1,\dots,J_n)\subseteq \mathbb R^n$ the closed
unbounded convex polyhedron defined as follows. An element
$s=(s_1,\dots,s_n)\in \mathbb R^n$ belongs to $\mathcal P'(J_1,\dots
,J_n)$ if and only if the inequalities
$$
s_i\geq d(0,J_i)\,, \qquad i=1,\dots,n\,,
$$
are satisfied. Note that if $J_1,\dots ,J_n$ are closed intervals
then the triangle inequality immediately implies $\mathcal
P'(J_1,\dots,J_n)\subseteq \mathcal P(J_1,\dots,J_n)$.

In the case of a non orientable surface we fix $F_N'\subseteq F_N$
choosing for any two dimensional cell arbitrarily one among the two
possible orientations. Note that in this case is not possible to
select the orientations in such a way that any pair of adjacent
elements of $F_N'$ is oriented in agreement. Likewise we fix
$E_N'\subseteq E_N$ choosing arbitrarily one among the two possible
orientations for any 1 dimensional cell.

For a non orientable surface, fixed $(x,y)\in E_N'$, two possible
situations are possible. The first case is when there exist one
element $f_+\in F_N'$ such that $(x,y)\in f_+$ and one $f_-\in F_N'$
such that $(y,x)\in F_N'$. When this happens we define
$J(x,y):=I(x,y)$ as for orientable surfaces. The second case is when
either do not exist elements of $F_N'$ to which $(x,y)$ belongs or
there are $f_1,f_2\in F_N'$ associated to adjacent two dimensional
cells such that $(x,y)\in f_i$, $i=1,2$. We then define
$$
J(x,y):=\Big(-\infty,
\min\left\{\psi(f_1),\psi(f_2)\right\}\Big]\cup\Big[\max
\left\{\psi(f_1),\psi(f_2)\right\},+\infty\Big)\,,
$$
where depending on the cases either $(y,x)\in f_i$ or $(x,y)\in
f_i$, $i=1,2$. We have the following Theorem.
\begin{Th}
\label{lisa} Consider a weighted oriented graph constructed starting
from a finite cellular subdivision of a compact non-orientable
surface without boundary and let $r\in W$ such that $\phi^r=\phi\in
d\Lambda^2$. Then we have $r\in R^e$ if and only if
$$
s\in \mathcal P'\left(\left\{J(x,y)\right\}_{(x,y)\in
E'_N}\right)\,.
$$
\end{Th}

\subsection{Applications}

Finally we discuss some applications. In particular we concentrate
on the elementary decomposition for the two dimensional torus. We
start from a smooth continuous vector field on $\mathbb T^2$ and
consider its Hodge decomposition. The natural continuous counterpart
of a discrete vector field on $d\Lambda^2$ is a vector field
obtained as the orthogonal gradient of a smooth potential function
$\psi$. We introduce then a natural discretization procedure. The
value of the discrete vector field on the edge $(x,y)\in E_N$ is the
value of the flux of the continuous vector field across the dual
edge of $(x,y)$ (see Remark \ref{duality} for the definition) with a
normal vector oriented in agreement with $(x,y)$. The result of the
discretization procedure is an element of $d\Lambda^2$. Then we
apply Theorem \ref{ugu} to this very general example  and obtain a
condition for the validity of the elementary decomposition in terms
of the variation of the potential function $\psi$.

Using the above framework we then construct a periodic random
environment on $\mathbb Z^d$. Theorem \ref{ugu} gives a condition on
the strength of the noise to be added in such a way that an
elementary decomposition holds almost surely. Using the results in
\cite{D} we can then deduce a quenched Central Limit Theorem.

All the applications are discussed in an informal way but the claims could be easily transformed into Theorems.

\section{Preliminary notions and results}
\label{qui} We start recalling some classic definitions of convex
analysis. See for example \cite{G} for more details.

\smallskip
Given $w^1,\dots,w^k$ distinct elements of $\mathbb Z^d$, with
$2\leq k\leq d+1$, we will say that they are in general position if
the vectors
\begin{equation}
\left\{ \begin{array}{l}
 d^1:=w^2-w^1\,\\
\vdots\\
 d^{k-1}:=w^{k}-w^1\,
\end{array}
\right.\label{defd}
\end{equation}
are linearly independent. By convention a single vector $w^1\in
\mathbb Z^d$ will be always considered to be in general position. It
is easy to see that this definition does not depend on the specific
order among the vectors.

\smallskip

Given $A\subseteq \mathbb R^d$ we denote by $A^0$ its relative
interior part. This is defined as
$$
A^0:=\left\{x\in A\,:\, \exists \epsilon >0 \ s.t.\
B_\epsilon(x)\cap aff(A) \subseteq A\right\}\,,
$$
where $B_\epsilon(x)$ is the Euclidean ball of radius $\epsilon$
centered at $x$.

Given $\underline w:=(w^1,\dots, w^k)\in (\mathbb Z^d)^k$ in general
position then $co(\{w^1,\dots,w^k\})$ is a $(k-1)-$dimensional
simplex and consequently for any $y\in co(\{w^1,\dots,w^k\})$ there
exists a unique  element $\mu:=(\mu_1,\dots,\mu_k)$ of
\begin{equation}
\mathbb S^{k}:=\left\{\mu=(\mu_1,\dots ,\mu_k)\ :\ \mu_i\geq 0\ ,\
\sum_{i=1}^k\mu_i=1\right\}\,,
\nonumber
\end{equation}
such that
\begin{equation}
\sum_{i=1}^k\mu_iw^i=y\,. \label{simplex}
\end{equation}
When $y\in \left(co(\{w^1,\dots,w^k\})\right)^0$ the coefficients
$\mu_i$ satisfy in addition the strict inequalities $0<\mu_i< 1$.

We recall that on $\mathbb R^d$, as shown in \cite{RW}, $co(A)$
coincides with the set of elements $x\in\mathbb R^d$ that can be
written as $x=\int_{\mathbb R^d}yd\mu(y)$, where $\mu$ is any
probability measure such that there exists a measurable set
$A'\subseteq A$ such that $\mu(A')=1$. In general this equivalence
is false.

Let us recall the following basic result of convex analysis (see for
example \cite{G})
\begin{Th}{\bf [Carath\'{e}odory]}
Let $A\subseteq \mathbb R^d$. Then for any $x\in co(A)$ there exist
$x^1,\dots ,x^k\in A$ in general position such that $x\in
co(\{x^1,\dots,x^k\})^0$. \label{kaka}
\end{Th}
Remember that since the vectors are in general position then
necessarily $k\leq d+1$.

We state and prove the following characterization of the convex hull
that we could not find in the literature.
\begin{Le}
\label{dueth} Let $x\in \mathbb R^d$ and $S\subseteq \mathbb R^d$.
Then $x\in co(S)$ if and only if for any hyperplane $H\ni x$ it
holds
\begin{equation}
H^+\cap S\neq \emptyset\qquad and \qquad H^-\cap S\neq \emptyset\,.
\label{duequaz}
\end{equation}
\end{Le}
\begin{proof}
Note that it holds $\left(H^+\cap S\right)\cup\left(H^-\cap
S\right)=S$. First we suppose that $x\in co(S)$ and show that
\eqref{duequaz} holds. Assume by contradiction that for example
there exists $H$ such that $H^+\cap S=S$ and $H^-\cap S=\emptyset$.
This implies also that $H^-\cap co(S)=\emptyset$. As a consequence
we have
$$
x=x\cap co(S)\subseteq H^-\cap co(S)=\emptyset\,,
$$
a contradiction.

Conversely we assume that \eqref{duequaz} holds for any hyperplane
$H$ and show that this implies $x\in co(S)$. Assume by contradiction
that $x\not\in co(S)$. Then (see Theorem 4.4 in \cite{G}) there
exists a separating hyperplane $\tilde{H}$ among the two disjoint convex sets
$\left\{x\right\}$ and $co(S)$ for which it holds for example
$\tilde{H}^+\cap S=\emptyset$ and $x\in \tilde{H}^+$. If $H$ is the
hyperplane parallel to $\tilde{H}$ and containing $x$ then we have
$H^+\subseteq \tilde{H}^+$ that implies $H^+\cap S=\emptyset$, a
contradiction.
\end{proof}

We recall the following basic result of convex analysis
(see for example \cite{G})
\begin{Th}{\bf [Helly]}
\label{Helly} Consider $\left\{I_\alpha\right\}_{\alpha \in \mathcal
A}$ a collection of compact convex subsets of $\mathbb R^d$. It
holds $\bigcap_{\alpha \in \mathcal A}I_\alpha\neq \emptyset$ if and
only if for any $(\alpha_1,\dots,\alpha_{d+1})\in \mathcal A^{d+1}$
it holds $\bigcap_{i=1}^{d+1}I_{\alpha_i}\neq \emptyset$.
\end{Th}

We defined the vector spaces $\Lambda^0, \Lambda^1, \Lambda^2$
associated to the discrete torus but it is clear that such vector
spaces can be defined for any cellular complex. Correspondingly the
following general result holds (see for example \cite{DKT},
\cite{daf} and \cite{Me}).
\begin{Th}{\bf [Discrete Hodge Decomposition]}
For any finite cellular complex it holds the following orthogonal
decomposition
\begin{equation}
\Lambda^1=\delta \Lambda^0\oplus d\Lambda^2\oplus \Lambda^1_H\,,
\nonumber
\end{equation}
where $\Lambda^1_H$ is called the subspace of harmonic one forms.
\label{DHD}
\end{Th}
We briefly discuss the ideas behind the proof of this Theorem
considering the case of the 2 dimensional torus. The elements of
$$
\delta \Lambda^0:=\left\{\phi\in \Lambda^1\ :\ \exists f\in
\Lambda^0\ s.t.\ \phi=\delta f\right\}\,,
$$
are called potentials or gradient vector fields. The dimension of
the vector space $\Lambda^0$ is $|V_N|$. It is easy to see that the
kernel of the co-boundary operator $\delta$ on $\Lambda^0$ coincides
with the constant functions, and in particular is a subspace of
dimension $1$. By the general identity
\begin{equation}
dim\left(\Lambda^0\right)=dim\left(Ker(\delta_{|\Lambda^0})\right)+dim\left(\delta
\Lambda^0\right)\,, \label{ker}
\end{equation}
we deduce that $dim(\delta \Lambda^0)=|V_N|-1=N^2-1$, where
$dim(\cdot)$ denotes the dimension. The orthogonal complement of
$\delta \Lambda^0$ in $\Lambda^1$ is easily characterized. The
subspace $\delta \Lambda^0$ of $\Lambda^1$ is spanned by the
elements $\left\{\delta\, \mathbb I_x\right\}_{x\in V}$, where
$\mathbb I_x$ is the characteristic function of $x\in V_N$. We
deduce that an element $\phi\in \Lambda^1$ belongs to the orthogonal
complement of $\delta \Lambda^0$ if and only if for any $x\in V_N$
it holds
\begin{equation}
\langle\phi,\delta \,\mathbb I_x\rangle_{1}=-d
\phi(x)=0\,.
\nonumber
\end{equation}
This means that the orthogonal complement of $\delta \Lambda^0$ is
the set of divergence free discrete vector fields also called
circulations.

We have also that $dim(\Lambda^2)=|F'_N|=N^2$. It is easy to see
that the kernel of the boundary operator $d$ on $\Lambda^2$ is the
one dimensional subspace of the constant 2 forms
$$
Ker(d_{|\Lambda^2})=\left\{\psi\,:\, \psi(f)=c\,,\, \forall f\in
F_N'\,,\, c\in \mathbb R \right\}\,.
$$
By the formula analogous to \eqref{ker} we deduce
$dim(d\Lambda^2)=N^2-1$. To show a part of the orthogonal
decomposition in Theorem \ref{DHD} it is then enough to show that
$d\psi$ is divergence free for any $\psi\in \Lambda^2$. This is
exactly the content of formula \eqref{dd=0}. Elements of
$d\Lambda^2$ are called $0-$homologous circulations. The
orthogonal complement $(d\Lambda^2)^\bot$ is characterized as
follows. Given $g\in F'_N$ we define $\psi_g\in \Lambda^2$ as
\begin{equation}
\psi_g(f):=\left\{
\begin{array}{ll}
+1 & \hbox{if} \ f=g\,,\\
-1 & \hbox{if} \ f=g^c\,,\\
0 & \hbox{otherwise}\,.\\
\end{array}
\right.\nonumber
\end{equation}
Clearly $d\Lambda^2$ is spanned by $\left\{d\psi_g\right\}_{g\in
F'_N}$. An element $\phi\in \Lambda^1$ belongs to
 $(d\Lambda^2)^\bot$ if and only if for any $g\in F'_N$ it holds
\begin{equation}
\langle \phi,d\psi_g\rangle_{1}= \delta \phi(g)=0\,. \label{rf}
\end{equation}
This means that the rotation of $\phi$ around any $g\in F'_N$ is
zero. An element $\phi\in \Lambda^1$ that satisfies condition
\eqref{rf} for any $g\in F'_N$ is called rotation free. Clearly any
gradient vector field is rotation free.

Finally if we define
\begin{equation}
\Lambda^1_H:=(\delta \Lambda^0)^\bot\cap(d \Lambda^2)^\bot\,,
\label{har}
\end{equation}
the orthogonal decomposition is proved. By \eqref{har} we have that
the elements of $\Lambda^1_H$ are rotation free circulations. Such a
kind of discrete vector fields are called  harmonic. From the
dimensional counting we have $dim\left(\Lambda^1_H\right)=2$.

Let us define $\phi_i\in \Lambda^1$, $i=1,2$, as
\begin{equation}
\phi_i(x,y):=\left\{
\begin{array}{ll}
+1 & \hbox{if} \ y=x+e^{(i)}/N\\
-1 & \hbox{if} \ y=x-e^{(i)}/N\\
0 & \hbox{otherwise}\,.
\end{array}
\right. \label{B1}
\end{equation}
As it is easy to check  $\phi_1$ and $\phi_2$ are linearly
independent rotation free circulations. Since
$dim\left(\Lambda^1_H\right)=2$ we can identify
\begin{equation}
\Lambda^1_H=\left\{c_1\phi_1+c_2\phi_2\ ,\ c_i\in \mathbb
R\right\}\,. \nonumber
\end{equation}

\medskip

Consider a square matrix $M$ whose rows and columns are labeled by a
finite set $V$ and call $M(x,y)$ the element corresponding to row
$x$ and column $y$. The matrix $M$ is called bi-stochastic if its
elements are non negative and moreover
$$
1=\sum_{y'\in V}M(x,y')=\sum_{x'\in V}M(x',y)\,, \qquad \forall
x,y\in V\,.
$$
These conditions identify a convex compact subset called the
Birkhoff polytope.  To every element $\pi\in Sym(V)$, the
permutation group on $V$, we can associate a $|V|\times |V|$ matrix
$M_\pi$ called the permutation matrix. It is defined as
\begin{equation}
M_\pi(x,y):=\left\{
\begin{array}{ll}
1 & \hbox{if} \ y=\pi(x)\,,\\
0 & \hbox{otherwise}\,.
\end{array}
\right.\label{permaflex}
\end{equation}
It is clear that $|V|\times |V|$ matrices having positive elements
and whose rows and columns are labeled by elements of $V$, are in
bijection with weighted oriented graphs on $V$. The bijection
identifies a matrix $M$ and a graph $(V,r)$ when $M(x,y)=r(x,y)$ for
any $x,y\in V$. For example the weighted graph corresponding to
\eqref{permaflex} has weights $r^\pi(x,y):=1$ if $y=\pi(x)$ and zero
otherwise. The following result is classic (see for example
\cite{G}).
\begin{Th}{\bf [Birkhoff-Von-Neumann]}
The set of $|V|\times |V|$ bi-stochastic matrices is convex and
compact. Its extremal elements are the permutation matrices on $V$.
\label{BVN}
\end{Th}

\section{Cyclic random walks on $\mathbb Z^d$}
\label{santorogo}
Recall the definition of irreducible equivalence class of cycles stated just before Theorem \ref{teoesterm}.
The following Lemma identifies an important class of irreducible equivalence classes of cycles.

\begin{Le}
Consider $\underline w=(w^1,\dots,w^k)\in (\mathbb Z^d)^k$ in
general position and such that $0\in co(\{w^1,\dots,w^k\})^0$. Then
there exists an unique collection of strictly positive natural
numbers $n_1,\dots,n_k$ such that
$$[C]=\{(w^1,n_1),\dots,(w^k,n_k)\}$$
is an irreducible element of $\mathcal C$. We will call $\mathcal
C^*$ the set of irreducible elements of $\mathcal C$ obtained in
this way. \label{lemmaestrem}
\end{Le}
\begin{proof}
If $y\in co(\{w^1,\dots,w^k\})^{0}\cap \mathbb Z^d$ then the element
$\mu \in \mathbb S^k$, uniquely determined by \eqref{simplex} has
the coordinates $\mu_i$ that are rational numbers. Indeed if we call
$A$ the $(d+1)\times k$ matrix whose $i$-column is $(w^i_1,
\dots,w^i_d,1)$, then $A$ has rank $k$. Every $k\times k$ sub-matrix
containing the last row has determinant different from zero. This
follows easily by the fact that the vectors are in general position.
Moreover $\mu$ is the unique solution of the linear system of $d+1$
equations in the $k$ variables $\mu_i$
\begin{equation}
\sum_{j=1}^kA_{i,j}\mu_j=\delta_{i,d+1}+(1-\delta_{i,d+1})y_i\,, \ \
\ \ \ i=1,\dots, d+1\,,\label{sist}
\end{equation}
where $\delta_{i,j}$ is the Kronecker delta. The unique solution to
\eqref{sist} can be obtained by Cramer formula applied to $k$
equations among which the last one. The matrices to be used have
integer coefficients so that the solution is a vector of rational
numbers. By construction they are also strictly positive. This means
that if we fix $y=0$ in \eqref{simplex} we obtain rational values
for the $\mu_i$. Let us write $\mu_i=\frac{a_i}{b_i}$ where $a_i$
and $b_i$ are natural numbers and $\frac{a_i}{b_i}$ is an
irreducible fraction for every $i$. Let $b:=lcm \{b_1,\dots,b_k\}$
the least common multiple and define the natural numbers
\begin{equation}
n_i:=b\mu_i\,, \ \ \ \ \  i=1, \dots, k\,.
\nonumber
\end{equation}
We have
\begin{equation}
\sum_{i=1}^kn_iw^i=b\sum_{i=1}^k\mu_iw^i=0\,. \label{prima}
\end{equation}
We show now that
\begin{equation}
\sum_{i=1}^km_i w^i\neq 0\,, \label{dorme}
\end{equation}
when the strict inclusion
$$
\{(w^1,m_1),\dots,(w^k,m_k)\}\subset
\{(w^1,n_1),\dots,(w^k,n_k)\}\,,
$$
holds. Indeed, if \eqref{dorme} is false then
\begin{equation}
\sum_{i=1}^k\frac{m_i}{(\sum_{j=i}^km_j)} w^i= 0\,, \label{dorme2}
\end{equation}
and since the vectors $w^i$ are in general position we deduce
\begin{equation}
\frac{m_i}{\sum_{j=i}^km_j}=\mu_i=\frac{a_i}{b_i}\,. \label{gelato}
\end{equation}
This implies
$$
lcm\{b_1,\dots,b_k\}\leq\sum_{j=i}^km_j<\sum_{j=i}^kn_j=b=lcm\{b_1,\dots,b_k\}\,,
$$
and we get a contradiction. Equations \eqref{prima} and
\eqref{dorme} imply that if we define
$$
[C]:=\left\{(w^1,n_1), \dots,(w^k,n_k)\right\}\,,
$$
then $[C]\in \mathcal C$ and moreover it is irreducible.

It remains to show that the numbers $n_i$ are uniquely
characterized. Let us suppose that $\{(w^1,m_1),\dots,(w^k,m_k)\}$
is an irreducible element of $\mathcal C$, then we want to show that
necessarily $m_i=n_i$. Clearly for such a collection of integer
numbers $m_i$, equations \eqref{dorme2} and \eqref{gelato} hold. We
deduce $\sum_{j=i}^km_j=lb$, with $l\in \mathbb N$ and $l\geq 1$.
When $l\geq 2$ then $\{(w^1,m_1),\dots,(w^k,m_k)\}$ is not
irreducible. When $l=1$ then $m_i=n_i$ and we obtain $[C]$.
\end{proof}

Not all irreducible equivalence classes of cycles belong to
$\mathcal C^*$. Indeed consider
$$
w^1:=(-5,-5)\,,\ w^2:=(1,1)\,,\  w^3:=(2,2)\,,
$$
vectors in $\mathbb Z^2$ and the equivalence class $\left\{(w^1,1),
(w^2,1), (w^3,2)\right\}$. It is easy to see that such equivalence
class is irreducible nevertheless it does not belong to $\mathcal
C^*$.

\smallskip

\noindent {\it Proof of Theorem \eqref{florisa}.} First we show that
if $p\in \mathcal B$ then it is cyclic. Since $p\in\mathcal B$ it
holds \eqref{twosided} for any $H$. Equation \eqref{twosided} can
hold only if \eqref{duequaz} holds with $S=\mathcal S(p)$. By Lemma
\ref{dueth} we deduce that $0\in co(\mathcal S(p))$. We can then
apply Carath\'{e}odory's Theorem \ref{kaka} and deduce that there
exists $\underline w=(w^1,\dots,w^k)$ in general position with
$w^i\in \mathcal S(p)$ and such that $0\in co(\{w^1,\dots,w^k\})^0$.
Let $[C]\in \mathcal C^*$ be the corresponding irreducible
equivalence class of cycles as constructed in Lemma
\ref{lemmaestrem}. The corresponding purely cyclic measure is
\begin{equation}
p^{[C]}=\sum_{i=1}^k\mu_i(\underline w)\delta_{w^i}:=q_1\,.
\label{defq}
\end{equation}
Here and hereafter we call $\mu(\underline w)$ the unique element of
$\mathbb S^k$ determined by \eqref{simplex} with $y=0$ when the
elements of $\underline w$ are in general position and moreover
$0\in co\left(\{x^1,\dots ,x^k\}\right)^0$.

Let
\begin{equation}
m_1:=\min\left\{\frac{p(w^i)}{q_1(w^i)}\,,\,
i=1,\dots,k\right\}>0\,. \label{m1}
\end{equation}
We have that $m_1q_1\preceq p$. This implies that $p_1:=p-m_1q_1$
belongs to $\mathcal M^{\leq 1}$ and moreover, since $q_1$ is a mean
zero probability measure, $p_1$ is balanced. Note
also that by construction there exists $x\in \mathcal S(p)$ such
that $p_1(x)=0$. Such an $x$ is a vector $w^i$ that minimizes
\eqref{m1}.

We want now iterate this procedure. More precisely let $p_i$ be a balanced measure.
Then  $0\in co(\mathcal S(p_i))$. As before
we can apply Carath\'{e}odory's Theorem identifying vectors
$\underline w=(w^1,\dots, w^k)$ in general position, determine the
corresponding element of $\mathcal C^*$ and define like in
\eqref{defq} the corresponding purely cyclic measure $q_{i+1}$. Then
we define
\begin{equation}
m_{i+1}:=\min\left\{\frac{p_i(w^j)}{q_{i+1}(w^j)}\,,\,
j=1,\dots,k\right\}>0\,, \label{recurenzo}
\end{equation}
and finally call
\begin{equation}
p_{i+1}:=p_i-m_{i+1}q_{i+1}\,, \label{recurenza}
\end{equation}
that is still a balanced element of $\mathcal M^{\leq 1}$. When $|\mathcal
S(p)|<+\infty$, after a finite number $l$ of iterations of the above
procedure we obtain $p_l=0$. This follows directly by the fact that
$|\mathcal S(p_{i+1})|\leq |\mathcal S(p_{i})|-1$. As a consequence
we get
\begin{equation}
p=\sum_{i=1}^lm_iq_i\,.\label{finfin}
\end{equation}
Recalling that the probability measures $q_i$ are purely cyclic,
equation \eqref{finfin} is exactly a representation
of $p$ of the type
\begin{equation}
p=p^\rho=\sum_{[C]\in\mathcal C^*}\rho([C])p^{[C]}\,, \label{pollo}
\end{equation}
where the measure $\rho$ gives weight $m_i$ to the unique element of
$\mathcal C^*$ associated to $q_i$, $i=1,\dots,l$.

When $|\mathcal S(p)|=+\infty$ we need to implement the iterative
procedure in a suitable way. Let us introduce the family of cubes
\begin{equation}
\Lambda_n:=\left\{x\in \mathbb Z^d\,:\, \max_i |x_i|\leq n
\right\}\,.
\nonumber
\end{equation}
Given a balanced $p_i$ we define
$$
n_i:=\inf\left\{n\in \mathbb N\,:\, 0\in co\left(\mathcal S(p_i)\cap
\Lambda_n\right)\right\}\,.
$$
Clearly $n_i<+\infty$, since $0\in co(\mathcal S(p_i))$ and
consequently there exists a finite number of elements in $\mathcal
S(p_i)$ whose convex envelope contains the origin. We can then
define the measure $q_{i+1}$ as in \eqref{defq} using vectors
$\underline w$ in general position and belonging to
$\Lambda_{n_i}\cap \mathcal S(p_i)$. Defining $p_{i+1}$ like in
\eqref{recurenzo} and \eqref{recurenza} we have $\mathcal
S(p_{i+1})\subset \mathcal S(p_{i})$ from which we deduce
$n_{i+1}\geq n_i$. We obtain in this way an increasing family of
cubes $\Lambda_{n_i}\subseteq \Lambda_{n_{i+1}}$ and a decreasing
family of measures $p_{i+1}\preceq p_i$. We now show that
necessarily
\begin{equation}
\left\{
\begin{array}{l}
\lim_{i\to +\infty}\Lambda_{n_i}=\mathbb Z^d\,, \\
\lim_{i\to +\infty}p_i=0\,,
\end{array}
\right.\label{limitlimit}
\end{equation}
where in the second limit it is enough to show just pointwise
convergence. As a consequence, taking the limit $j\to +\infty$ in
$$
p_j=p-\sum_{i=1}^jm_iq_i\,,
$$
we obtain
\begin{equation}
p=\sum_{i=1}^{+\infty}m_iq_i\,.
\label{finfin1}
\end{equation}
This identifies $p$ with $p^{\rho}$ like in \eqref{pollo}, where the
probability measure $\rho$ on $\mathcal C^*$  gives weight $m_i$ to
the unique element of $\mathcal C^*$ associated to $q_i$.

We need then to show \eqref{limitlimit}. Both sequences are monotone
and then the limits exist. Let us suppose by contradiction that
$\lim_{i\to +\infty}n_i=n^*<+\infty$. Note that in this case for any
$i$ we have $\mathcal S(q_i)\subseteq \Lambda_{n^*}$ and moreover
$$
|\mathcal S(p_{i+1})\cap\Lambda_{n^*}|\leq |\mathcal
S(p_{i})\cap\Lambda_{n^*}|-1\,.
$$
In particular after a finite number of iterations, say $j$ we have
that $\mathcal S(p_j)\cap \Lambda_{n^*}=\emptyset$. This implies
$n_{j+1}>n^*$ and we get a contradiction.

Finally let us suppose by contradiction that $\lim_{i\to
+\infty}p_i=p^*\neq 0$. Then, by monotone convergence Theorem,
$p^*$ is balanced and consequently $0\in co(\mathcal S(p^*))$. Let
us call $w^1,\dots ,w^k$ a finite collection of vectors in general
position belonging to $\mathcal S(p^*)$ and such that $0\in
co(\{w^1,\dots,w^k\})^0$. Consider $j\in \mathbb N$ such that
$\{w^1,\dots,w^k\}\subseteq \Lambda_{n_j-1}$. By construction we
have $\mathcal S(p^*)\subset \mathcal S(p_j)$ and moreover $0\not
\in co\left(\mathcal S(p_j)\cap\Lambda_{n_j-1}\right)$. This is a
contradiction.
\smallskip

We now prove that if $p$ is cyclic then $p\in \mathcal B$. By
hypothesis (recall \eqref{igora} we have
\begin{equation}
p=\lim_{n \to +\infty}\sum_{[C]\in \mathcal C_n}\rho([C])p^{[C]}\,,
\label{nespoli}
\end{equation}
and for any fixed $n$ the sum on the r.h.s. of \eqref{nespoli}
satisfies \eqref{twosided} for any $H$. We get immediately the
validity of \eqref{twosided} for $p$ by applying the monotone
convergence Theorem. \qed

\medskip

\noindent {\it Proof of Theorem \eqref{teoesterm}.} The extremality
of $[C]\in \mathcal C^*$ follows directly by the irreducibility. The
validity of \eqref{airberlin} has been implicitly \ proved during
the proof of Theorem \ref{florisa}.\qed

\medskip

\noindent {\it Proof of Theorem \eqref{procbreve}.} To show the
identification it is enough to show that $\overrightarrow{\mathcal
M_0}\cap\mathcal M^1$ coincides with the set of cyclic measures.

Let $p$ be a cyclic probability measure, then clearly $p\in \mathcal
M^1$. Moreover it holds \eqref{nespoli} and this is the monotone limit assuring $p\in
\overrightarrow{\mathcal M_0}$.

Conversely take $p\in \overrightarrow{\mathcal M_0}\cap\mathcal
M^1$, we need to show that $p$ is cyclic. By definition $p\in
\mathcal M^1$ and there exists a non decreasing sequence $p^n\in
\mathcal M_0$ such that $p=\lim_{n\to +\infty}p^n$. Since $p^n\in
\mathcal M_0$ then also $\Delta^n:=p^n-p^{n-1}\in \mathcal M_0$,
where we defined $p^0:=0$. This means that we can write
$\Delta^n=\sum_{[C]\in \mathcal C^*}\rho^n([C])p^{[C]}$ for suitable
measures $\rho^n$. This follows by the fact that any element of
$\mathcal M_0$ is balanced and then by Theorem \ref{florisa} is
cyclic. Since we have $p=\sum_{n=1}^{+\infty}\Delta^n$ we get
\eqref{decomp} with $\rho([C])=\sum_{n=1}^{+\infty}\rho^n([C])$ and
$p$ on the l.h.s.. \qed

\section{Cyclic random walks on a finite set}
\label{crwfg}

The following results are elementary but useful for the forthcoming
results.
\medskip

{\it Proof of Theorem \ref{teofin}}
Any purely cyclic graph is balanced. This implies that a necessary
condition for the validity of \eqref{decgr} is that $r$ is balanced.
Let us now show the other implication. Let
$$
m^*_1:=\min_{(x,y)\in
E(r)}r(x,y)>0
$$
and let $(z^0,z^1)\in E(r)$ such that $r(z^0,z^1)=m_1^*$. Due to the
balancing condition \ref{bal222} and the definition of $m_1^*$ there
exists an edge $(z^1,z^2)\in E(r)$ with $r(z^1,z^2)\geq m_1^*$. By
the same argument, if $z^2\neq z^0$, there exists $(z^2,z^3)\in
E(r)$ such that $r(z^2,z^3)\geq m_1^*$. We can iterate this
procedure up to the first time we visit twice a vertex of $V$. Since
$V$ is finite this happens after at most $|V|$ iterations. We obtain
in this way a sequence $z^0,z^1, \dots, z^{n-1}$ of distinct
elements of $V$ such that $(z^i,z^{i+1})\in E(r)$ and moreover a
$z^n$ such that $z^n=z^j$ for some $0\leq j\leq n-1$. We call $C_1$
the cycle $C_1:=(z^j,z^{j+1},\dots, z^n)$. We also call
$m_1:=\min_{i=j,\dots,n-1}r(z^i,z^{i+1})\geq m_1^*>0$. Clearly we
have $(V,m_1r^{[C_1]})\preceq (V,r)$.  As a consequence we have that
$r-m_1r^{[C_1]}\in W$, it is still balanced and moreover it holds
$|E\left(r-m_1r^{[C_1]}\right)|\leq E(r)-1$. This last inequality
implies that after a finite number (at most $|E(r)|$) of iterations
of the above procedure we obtain
$$
r=\sum_{i=1}^lm_ir^{[C_i]}\,,
$$
that is the decomposition \eqref{decgr} with the measure $\rho$ that
gives weight $m_i$ to $[C_i]\in \mathcal C$ \qed
\begin{Rem}
The validity of the balancing condition \ref{bal222} implies that
the uniform measure $\pi(x):=\frac{1}{|V|}$ satisfies the stationary
condition
\begin{equation}
\pi(x)\sum_{\left\{y\,:\,(x,y)\in
E(r)\right\}}r(x,y)=\sum_{\left\{y\,:\,(y,x)\in
E(r)\right\}}\pi(y)r(y,x)\,.
\nonumber
\end{equation}
If the Markov chain is irreducible, this is the unique invariant measure.
\end{Rem}
In the case of a continuous time Markov chain the set of balanced
rates is convex but not compact. In the case of discrete time Markov
chains, in addition to the balancing condition there are also the
conditions \eqref{norm}. The constraints $r\geq 0$, \eqref{norm} and
\eqref{balancing} for every $x$, identify the Birkhoff polytope. The
extremal elements of the Birkhoff polytope are then characterized by
the Birkhoff-Von-Neumann Theorem \ref{BVN}. Theorem \ref{BVN}
together with the classical statement of Krein-Milmann Theorem
\cite{G}, \cite{P} implies that, given any bi-stochastic matrix $M$,
we can decompose it like
\begin{equation}
M=\sum_{\pi \in Sym(V)}m_\pi M_\pi\,, \nonumber
\end{equation}
where $m$ is a probability measure on $Sym(V)$. Written in terms of
weights this equation becomes
\begin{equation}
r=\sum_{\pi\in Sym(V)}m_\pi r^\pi\,. \label{bafana}
\end{equation}
Recall the classical result that every permutation can be decomposed
into disjoint cycles. It is easy to see that in terms of weights
this means that for every $\pi \in Sym(V)$ we can write
\begin{equation}
r^\pi=\sum_ir^{[C_i^\pi]}\,, \label{decper}
\end{equation}
where the $C_i^\pi$ constitutes a family of disjoint cycles such
that every element of $V$ belongs to one of them. The cycle
containing the element $x^0\in V$ can be written as
$(x^0,\pi(x^0),\pi^2(x^0),\dots,\pi^l(x^0), x^0)$, where with
$\pi^m$ we denote the composition of $m-$times the element $\pi\in
Sym(V)$ and $l$ is the minimal integer such that
$\pi^{l+1}(x^0)=x^0$. Putting together \eqref{bafana} and
\eqref{decper} we obtain a special decomposition like \eqref{decgr}.

\smallskip

We finish the section discussing the case of an infinite graph. We
simply reinterpret the results of subsection \ref{ciclzd} in terms
of an infinite weighted graphs with vertices $V=\mathbb Z^d$ and
edges $E=\mathbb Z^d\times \mathbb Z^d$. The weighted graph
corresponding to the translation invariant Markov chain on
$V=\mathbb Z^d$ defined by \eqref{def}, gives weight $r(x,y)=p(y-x)$
to the edge $(x,y)$. Let us suppose that $p$ is a cyclic measure.
Consider $C_i$ a cycle representant of the equivalence class in
$\mathcal C^*$ corresponding to the cyclic measure $q_i$ in
\eqref{finfin1}. On $\mathbb Z^d$ there is defined a shift operator
$\tau_x$ that acts naturally on cycles by
$$
\tau_x(x^0,x^1,\dots,x^{n-1},x^0):=(x+x^0,x+x^1,\dots,x+x^{n-1},x+x^0)\,.
$$
Define the family of cycles $\left\{\tau_xC_i\right\}_{x\in \mathbb
Z^d}^{i\in \mathbb N}$ in $\mathbb Z^d$. Let $\rho$ be the positive
measure on $\mathcal C$ that gives weight $m_i$ to  $[\tau_xC_i]$
for any $x$. The results in subsection \ref{ciclzd} imply the
validity of the decomposition \eqref{decgr} with this specific
measure $\rho$. In this case \eqref{decgr} becomes
$$
r(x,y)=p(y-x)=\sum_{i\in \mathbb N}\sum_{z\in \mathbb
Z^d}m_ir^{[\tau_zC_i]}(x,y)\,, \ \ \ \ \forall\  (x,y)\,.
$$

\section{Cyclic random walks on a finite graph with topology}
\label{crwfgt} We discuss here the general case of a $d\geq 2$
dimensional torus. More detailed results for the specific cases
$d=1,2$ will be discussed separately in some subsections.

We start observing that formula \eqref{nasce} can be written as
\begin{equation}
r=r^{\phi^r}+\sum_{(x,y)\in E_N'} s(x,y)r^{[C_{\{x,y\}}]}\,.
\nonumber
\end{equation}

A subset $A\subseteq W$ is monotone non decreasing if $r\in A$ and
$r\preceq r'$ implies $r'\in A$. The subsets $R^e$ and $R^*$ are in
general not monotone subsets of $W$. Nevertheless $R^e\cap R(\phi)$
and $R^*\cap R(\phi)$ are non decreasing for any fixed $\phi$. This
is the content of the next lemma.
\begin{Le}
Consider $r$, and $r'$ belonging to $W$ and such that $r\preceq r'$
and $\phi^r=\phi^{r'}$. We have that if $r\in R^*$ then also $r'\in
R^*$. The same happens for $R^e$. \label{simpleimportant}
\end{Le}
\begin{proof}
We prove the statement for $R^*$. The proof for $R^e$ is the same.
By assumption there exists a decomposition like \eqref{dec*} for
$r$. We then have
\begin{equation}
r'=\sum_{[C]\in \mathcal C^*}\rho([C])r^{[C]}+ (r'-r)\,.
\label{bloes}
\end{equation}
By the hypotheses of the lemma $h:=r'-r$ belongs to $W$ and moreover
$\phi^h=0$. Clearly $R(0)\subseteq R^*$ and moreover for any $h\in
R(0)$ we have
\begin{equation}
h=\sum_{(x,y)\in E'_N}h(x,y)r^{[C_{\left\{x,y\right\}}]}\,.
\label{salamella}
\end{equation}
Putting together \eqref{bloes} and \eqref{salamella} we get $r'\in
R^*$.
\end{proof}

\noindent {\it Proof of Lemma \ref{neclass}} Since we discussed the
discrete Hodge decomposition in the two dimensional case we will prove
this Lemma also in the two dimensional case. The proof in the general
case is analogous.  Consider
$C=(z^0,z^1,\dots ,z^n)$ with $z^n=z^0$ a cycle such that $[C]\in
\mathcal C^*$. Since $C$ is homotopically trivial it holds \eqref{rivestimento}.
Componentwise \eqref{rivestimento} is
written as
$$
\left\{
\begin{array}{l}
|\{(x,x+e^{(1)}/N)\in C\}|-|\{(x+e^{(1)}/N,x)\in C\}|=0\,,\\
|\{(x,x+e^{(2)}/N)\in C\}|-|\{(x+e^{(2)}/N,x)\in C\}|=0\,,
\end{array}
\right.
$$
that implies
\begin{equation}
\left\{
\begin{array}{l}
\langle\phi^{r^{[C]}},\phi_1\rangle_{1}=0\,,\\
\langle\phi^{r^{[C]}},\phi_2\rangle_{1}=0\,,
\end{array}
\right.\nonumber
\end{equation}
where the vector fields $\phi_i$ are defined in \eqref{B1}. This
means that $\phi^{r^{[C]}}\in \left(\Lambda^1_H\right)^\bot$.
Moreover we know that $r^{[C]}$ satisfies the balancing condition
\ref{bal222}. This implies that $\phi^{r^{[C]}}$ is a divergence
free discrete vector field i.e. it is a circulation. This is
equivalent to say that it belongs to $(\delta \Lambda^0)^\bot$. By
Theorem \ref{DHD} we obtain $\phi^{r^{[C]}}\in d\Lambda^2$ for any
$[C]\in \mathcal C^*$. Given $r\in R^*$ having a decomposition like
\eqref{dec*}, by linearity of the projection of weights onto vector
fields we have
\begin{equation}
\phi^{r}=\sum_{[C]\in \mathcal C^*}\rho([C])\phi^{r^{[C]}}\,.
\label{arg}
\end{equation}
Since $d\Lambda^2$ is a vector subspace, it is closed under linear
combinations and consequently the right hand side of \eqref{arg}
belongs to $d\Lambda^2$. Since $\mathcal C^e\subseteq \mathcal C^*$
the condition $\phi^{r}\in d\Lambda^2$ is also necessary for the
validity of \eqref{decel}. \qed

\smallskip

\noindent {\it Proof of Lemma \ref{closed}} We prove the statement
for $R^*$. The proof for $R^e$ is the same. We need to prove that
for any sequence $r_n\in R^*$ converging to some element $r\in W$ we
have necessarily $r\in R^*$. Observe that for any $(x,y)\in E_N$ we
have $r_n(x,y)\geq 0$ and consequently $\lim_{n\to
+\infty}r_n(x,y)=r(x,y)\geq 0$. Since for every $n$ it holds $r_n\in
R^*$, we can write
$$
r_n=\sum_{[C]\in \mathcal C^*}\rho_n([C])r^{[C]}\,,
$$
for some positive measures $\rho_n$. Let $M:=\max_{(x,y)\in
E_N}r(x,y)$. Consider any $[C]\in \mathcal C^*$ and take $(x,y)\in
C$. For $n$ large enough we deduce
$$0\leq \rho_n([C])\leq\sum_{[C]\in \mathcal C^*}\rho_n([C])r^{[C]}(x,y)=r_n(x,y)\leq 2M\,.$$
The last inequality follows by the fact that $r_n$ converges to $r$
and $r(x,y)\leq M$. Since $\rho_n([C])$ takes values on a compact
set we can then extract a converging subsequence $\rho_{n_j}([C])$.
By a finite Cantor diagonalizing argument, there exists a
subsequence (that we still call $n_j$) such that $\rho_{n_j}([C])$
is converging for any $[C]\in \mathcal C^*$. Let us call
$\rho([C])\geq 0$ the corresponding limits. Since we have a finite
sum we get
$$
r=\lim_{j\to +\infty}r_{n_j}=\lim_{j\to +\infty} \sum_{[C]\in
\mathcal C^*}\rho_{n_j}([C])r^{[C]}= \sum_{[C]\in \mathcal
C^*}\rho([C])r^{[C]}\
$$
This is equivalent to say $r\in R^*$. \qed

\smallskip

The bi-dimensional example illustrated in figure \ref{fig2} shows
that the condition $\phi^{r}\in d\Lambda^2$ is not sufficient
neither for \eqref{dec*} nor for \eqref{decel}. Indeed in this case
$\phi^{r}\in d\Lambda^2$ but the oriented weighted graph $(V_N, r)$
contains no homotopically trivial cycles. Consequently
decompositions of the type \eqref{dec*} and \eqref{decel} are not
possible.

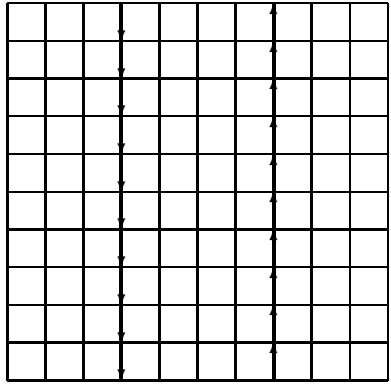
\begin{figure}
\setlength{\unitlength}{5cm}
\begin{picture}(1,1)
\multiput(0,0)(0.1,0){10}%
{\line(0,1){0.1}}
\multiput(0,0.1)(0.1,0){10}%
{\line(0,1){0.1}}
\multiput(0,0.2)(0.1,0){10}%
{\line(0,1){0.1}}
\multiput(0,0.3)(0.1,0){10}%
{\line(0,1){0.1}}
\multiput(0,0.4)(0.1,0){10}%
{\line(0,1){0.1}}
\multiput(0,0.5)(0.1,0){10}%
{\line(0,1){0.1}}
\multiput(0,0.6)(0.1,0){10}%
{\line(0,1){0.1}}
\multiput(0,0.7)(0.1,0){10}%
{\line(0,1){0.1}}
\multiput(0,0.8)(0.1,0){10}%
{\line(0,1){0.1}}
\multiput(0,0.9)(0.1,0){10}%
{\line(0,1){0.1}}

\multiput(0,0)(0.1,0){10}%
{\line(1,0){0.1}}
\multiput(0,0.1)(0.1,0){10}%
{\line(1,0){0.1}}
\multiput(0,0.2)(0.1,0){10}%
{\line(1,0){0.1}}
\multiput(0,0.3)(0.1,0){10}%
{\line(1,0){0.1}}
\multiput(0,0.4)(0.1,0){10}%
{\line(1,0){0.1}}
\multiput(0,0.5)(0.1,0){10}%
{\line(1,0){0.1}}
\multiput(0,0.6)(0.1,0){10}%
{\line(1,0){0.1}}
\multiput(0,0.7)(0.1,0){10}%
{\line(1,0){0.1}}
\multiput(0,0.8)(0.1,0){10}%
{\line(1,0){0.1}}
\multiput(0,0.9)(0.1,0){10}%
{\line(1,0){0.1}}
\multiput(0,1)(0.1,0){10}%
{\line(1,0){0.1}}
\multiput(1,0)(0,0.1){10}%
{\line(0,1){0.1}}

\linethickness{0.35mm}
\multiput(0.7,0)(0,0.1){10}%
{\vector(0,1){0.1}}
\multiput(0.3,1)(0,-0.1){10}%
{\vector(0,-1){0.1}}

\end{picture}
\caption{An $r\in W$ for the two dimensional torus such that
$\phi^r\in d\Lambda^2$ nevertheless $r\not \in R^*$. Positive
unitary weights are represented by boldfaced arrows. Opposite sides of the square
are identified.} \label{fig2}
\end{figure}

Recalling Lemma \ref{simpleimportant} and the fact that $r^\phi$ is
the minimal element of $R(\phi)$, it is natural to study the
following problem. Given $\phi\in d\Lambda^2$ a discrete vector
field, under which conditions $r^\phi\in R^*$, or $r^\phi\in R^e$?
When $r^\phi\in R^*$ then by Lemma \ref{simpleimportant} we deduce
that $R(\phi)\subseteq R^*$. When $r^\phi\in R^e$ then by Lemma
\ref{simpleimportant} we deduce that $R(\phi)\subseteq R^e\subseteq
R^*$. On the other side if $r^\phi\not\in R^*$ then by Lemma
\ref{closed} all the elements of $W$ in a neighborhood of $r^\phi$
will also not belong to $ R^*$. Note that inside such a neighborhood
there will be also rates such that $r(x,y)>0$ for any $(x,y)\in
E_N$. The same happens for $R^e$.

\medskip

\noindent {\it Proof of Corollary \ref{corclosed}} The proof follows
directly from the above argument and the fact that there exists a
$\phi\in d\Lambda^2$ such that $r^\phi\not\in R^*$. In dimension two
this is exactly the example in Figure \ref{fig2}. Similar examples
can be clearly constructed in any dimension $d\geq 2$.\qed

\smallskip

We will give a complete characterization of vector fields such that
$r^\phi\in R^e$ in dimension one and two. A characterization of
vector fields such that $r^\phi\in R^*$ seems to be an interesting
combinatorial problem.

\subsection{One dimensional torus}

The balancing condition \eqref{balancing} at site $x\in V_N$ can be
written as
\begin{equation}
r(x,x+1/N)-r(x+1/N,x)=r(x-1/N,x)- r(x,x-1/N)\,,
\nonumber
\end{equation}
and corresponds to require that the discrete vector field $\phi^r$
has divergence zero. Note that in one dimension a discrete vector
field is divergence free if and only if it is constant.

\smallskip

\noindent {\it Proof of Theorem \ref{teodoro}} We want to
characterize the Markov model having rates $r$ admitting a
decomposition of the type
\begin{equation}
r=\sum_{[C]\in \mathcal C^*}\rho([C])r^{[C]}=\sum_{x\in
V_N}\rho([C_{\left\{x,x+1/N\right\}}])
r^{[C_{\left\{x,x+1/N\right\}}]}\,. \label{decsh}
\end{equation}

If \eqref{decsh} holds then clearly we have that the associated
discrete vector field satisfies
$$
\phi^r(x,x+1/N)=\rho([C_{\left\{x,x+1/N\right\}}])
-\rho([C_{\left\{x,x+1/N\right\}}])=0\,.
$$
Conversely given a weight $r$ such that $\phi^r=0$ then we can set
in \eqref{decsh}
$\rho([C_{\left\{x,x+1/N\right\}}])=r(x,x+1/N)=r(x+1/N,x)$. Finally the
reversibility follows by the detailed balance condition
$$
\frac{1}{N}r(x,x+1/N)=\frac{1}{N}r(x+1/N,x)\,, \ \ \ \ \ \forall x\in
V_N\,,
$$
that is equivalent to $\phi^r=0$.\qed

\smallskip

\noindent {\it Proof of Theorem \ref{teodoro2}}
We want to characterize now the Markov models having a decomposition
of the type
\begin{equation}
r=\sum_{[C]\in \mathcal C}\rho([C])r^{[C]}=\sum_{x\in
V_N}\rho([C_{\left\{x,x+1/N\right\}}])
r^{[C_{\left\{x,x+1/N\right\}}]}+\rho([C_+])r^{[C_+]}+\rho([C_-])r^{[C_-]}\,.
\label{decsh2}
\end{equation}
The discrete one dimensional torus is a finite graph and
consequently this problem has been solved in Theorem \ref{teofin}.
We discuss it again in terms of the discrete vector field $\phi^r$.
The fact that in one dimension a discrete vector field has zero
divergence if and only if it is constant has already been stressed.
If \eqref{decsh2} holds then clearly we have
$$
\phi^r(x,x+1/N)=\rho([C_{\left\{x,x+1/N\right\}}])
-\rho([C_{\left\{x,x+1/N\right\}}])+
\rho([C_+])-\rho([C_-])=\rho([C_+])-\rho([C_-])\,,
$$
that is a constant discrete vector field. Conversely consider some
rates such that for any $x\in V_N$ it holds $\phi^r(x,x+1/N)=c$ with
$c$ a constant real number. Then from \eqref{decsh2} recalling the
decomposition \eqref{nasce} we obtain
\begin{equation}
\left\{
\begin{array}{l}
r(x,x+1/N)=\rho([C_{\left\{x,x+1/N\right\}}])+\rho([C_+])=[c]_++s(x,x+1/N)\,,\\
r(x+1/N,x)=\rho([C_{\left\{x,x+1/N\right\}}])+\rho([C_-])=[-c]_++s(x+1/N,x)\,,
\end{array}
\right. \label{trio}
\end{equation}
where we recall that $s$ is the symmetric part of $r$. Clearly
\eqref{trio} implies
\begin{equation}
\rho([C_{\left\{x,x+1/N\right\}}])=s(x,x+1/N)-a\,, \label{bio}
\end{equation}
where $a$ is a real parameter. Putting \eqref{bio} in \eqref{trio}
we obtain also
\begin{equation}
\left\{
\begin{array}{l}
\rho([C_+])=[c]_++a\,,\\
\rho([C_-])=[-c]_++a\,.
\end{array}
\right. \label{canz}
\end{equation}
Since the left hand sides of \eqref{bio} and \eqref{canz} are
non-negative we obtain the constraint $a \in [0,m]$. Recalling the
definition of the symmetric part $s$ we obtain \eqref{teonuovo}.
\qed

Note that the decomposition \eqref{decsh2} is unique only when
$m=0$.

\subsection{Two dimensional torus}

Let $\rho$ a positive measure on $\mathcal C^e$. The general
decomposition \eqref{decel} reads
\begin{equation}
r=\sum_{(x,y)\in
E_N'}\rho([C_{\left\{x,y\right\}}])r^{[C_{\left\{x,y\right\}}]}+\sum_{f\in
F_N}\rho([C_f])r^{[C_f]}\,. \label{lofa}
\end{equation}
To the measure $\rho$ in \eqref{lofa} we can also associate an
element $\psi^\rho\in\Lambda^2$ defined as
\begin{equation}
\psi^\rho(f):=\rho([C_f])-\rho([C_{f^c}])\, \qquad f\in F_N\,.
\label{lofa2}
\end{equation}

\begin{Le}
\label{questo}
If $r$ is like in \eqref{lofa} then we have $\phi^r=d\psi^\rho$.
\end{Le}
\begin{proof}
Given $(x,y)\in E_N$ it belongs to only two elements of $F_N$.
Moreover one is clockwise oriented and the other one is
anticlockwise oriented. Let us call $f_+$ the element of $F_N'$ such
that $(x,y)\in f_+$ and $f_-$ the element of $F_N'$ such that
$(x,y)\in f_-^c$. Then we have
$$
\phi^r(x,y)=\left(\rho([C_{f_+}])-\rho([C_{f_+^c}])\right)-
\left(\rho([C_{f_-}])-\rho([C_{f_-^c}])\right)=\psi^\rho(f_+)-\psi^\rho(f_-)\,.
$$
Recalling \eqref{defdd} and \eqref{lofa2} we obtain the statement of
the lemma.
\end{proof}

The following Remarks will not be used during our proofs but
will be useful in the following.
\begin{Rem}\label{duality}
({\it Duality}) We recall a well known duality relationship. To the
2 dimensional discrete torus $(V_N,E_N)$ we associate a dual
discrete torus $(\tilde V_N,\tilde E_N)$ defined as follows. The
vertices of the dual graph are the elements of $\mathbb T^2$ having
coordinates $x+\frac {1}{2N} (e^{(1)}+e^{(2)})$ with $x\in V_N$.
Note that every element of $\tilde V_N$ is the center of a cell of
the original cellular decomposition of $\mathbb T^2$. The set of
oriented edges $\tilde E_N$ is constituted by the pairs $(v,w)$ such
that $v,w\in \tilde V_N$ and moreover $d(v,w)=1/N$. It is possible
to define a duality map $D$. This map is defined both on $E_N$ with
image on $\tilde E_N$ and on $\tilde E_N$ with image in $E_N$. This
map is injective and satisfies the involution property $D^2=-\mathbb
I$. We use the same symbol $D$, both when it acts on $E_N$ or in
$\tilde E_N$, since it can be defined easily in the same way. Any
element $(x,y)$, both of $E_N$ and of $\tilde E_N$ can be naturally
represented by an arrow exiting from $x$ and entering in $y$. The
element $D(x,y)$ is defined as the unique element of, either $\tilde
E_N$ or $E_N$, whose representing arrow is obtained rotating
counterclockwise of $\frac{\pi}{2}$ the arrow representing $(x,y)$
around its middle point. The dual map can be naturally extended to
act on discrete vector fields. More precisely given for example
$\phi \in \Lambda^1$ we define $D\phi \in \tilde \Lambda^1$ as
$$
\left(D\phi\right)(w,z):=\phi(x,y)\,,
$$
where $(x,y)$ is the unique element of $E_N$ such that
$D(x,y)=(w,z)$. Clearly we called $\tilde \Lambda^1$ the set of
discrete vector fields on the dual torus. According to these
definition the notion of circulation and rotation free are dual to
each other. More precisely given a discrete vector field $\phi$
whose discrete divergence is zero in a vertex, then $D\phi$ will
have zero circulation around the corresponding dual face. Conversely
given a discrete vector field $\phi$ that satisfies the condition of
zero rotation around a face then $D\phi$ will satisfy the condition
of zero discrete divergence in the corresponding dual vertex.
\end{Rem}

\begin{Rem}
Given $\psi \in \Lambda^2$ it is simple to compute $\phi=d\psi$ by
\begin{equation}
\phi(x,y)=\psi(f_+)-\psi(f_-)\,.\label{inps}
\end{equation}
In terms of the dual graph it can be interpreted in the following
way. Consider $\tilde \psi$ the element of $\tilde{\Lambda}_0$ that
associates to any vertex of $\tilde V_N$ the value $\psi(f)$ where
$f\in F'_N$ is the oriented dual face of the fixed vertex. Then
\eqref{inps} can be written as
\begin{equation}
D\phi=\delta \tilde \psi\,.\label{gigia}
\end{equation}
In extended form \eqref{gigia} is
\begin{equation}
\left\{
\begin{array}{l}
\phi(x,x+e^{(1)}/N)=\delta \tilde\psi
(x+e^{(1)}/(2N)-e^{(2)}/(2N),x+e^{(1)}/(2N)+e^{(2)}/(2N))\,,\\
\phi(x,x+e^{(2)}/N)=-\delta \tilde\psi (x-e^{(1)}/(2N)+e^{(2)}/(2N),
x+e^{(1)}/(2N)+e^{(2)}/(2N))\,.
\end{array}
\right.\label{inpsd}
\end{equation}
Conversely given $\phi\in d\Lambda^2$, the determination of a $\psi$
such that $\phi=d\psi$ requires a non local computation. More
precisely using \eqref{gigia} we determine $\delta \tilde\psi \in
\tilde{\Lambda}_1$ and then we can compute
\begin{equation}
\psi(f)-\psi(f')=\sum_{i=1}^n\delta \tilde\psi(x^{i},x^{i+1})\,,
\label{iosoloio}
\end{equation}
where $(x^{i},x^{i+1})\in \tilde{E}_N$ and $x^{0}$ is dual to
$f'$ and $x^{n}$ is dual to $f$. \label{tusolotu}
\end{Rem}

\noindent {\it Proof of Theorem \ref{ugu}} Given $\phi\in
d\Lambda^2$ there exists a $\psi\in \Lambda^2$ such that $\phi=d
\psi$. Since the kernel of the boundary operator $d$ coincides with
the constant 2 forms, the elements $\psi'\in\Lambda^2$ such that
$d\psi'=\phi$ are exactly of the type $\psi+c$ where $c$ is an
arbitrary constant.

Let $r\in W$ having a decomposition like \eqref{lofa} and such that
$\phi^r=\phi$. By Lemma \eqref{questo} we have $\psi^\rho=\psi+c$
for some constant $c$. This means
\begin{equation}
\rho([C_f])-\rho([C_{f^c}])=\psi(f)+c\,, \qquad f\in
F'_N\,.\label{finiti}
\end{equation}
Recalling \eqref{nasce} we have that
\begin{equation}
s:=r-r^\phi\label{urp}
\end{equation}
is the symmetric part of $r$ and belongs to $R(0)$. We want to
characterize which are the elements $s\in R(0)$ that can be obtained
in \eqref{urp} when $r\in R^e\cap R(\phi)$.

The more general positive solution to \eqref{finiti} is
\begin{equation}
\rho([C_f])=[(\psi+c)(f)]_++m(f)\,,\qquad f\in F_N\,,
\label{ilsolito}
\end{equation}
where $m(f)$ are arbitrary non-negative numbers such that $m(f)=m(f^c)$.
In \eqref{ilsolito} $\psi+c$ is the element of $\Lambda^2$ obtained by the sum of
$\psi$ and the constant 2-chain $c\in \Lambda^2$ defined by $c(f)=c$ for any $f\in F_N'$. In particular
\eqref{ilsolito} means
$$
\rho([C_f])=[\psi(f)+c]_++m(f)\,,\qquad \rho([C_{f^c}])=[-\psi(f)-c]_++m(f)\,,\qquad f\in F_N'\,.
$$
Note that we can write
\begin{equation}
\sum_{f\in F_N}m(f)r^{[C_{f}]}=\sum_{(x,y)\in E'_N}
\left(\sum_{\left\{f\in F_N\,:\,(x,y)\in
f\right\}}m(f)\right)r^{[C_{\left\{x,y\right\}}]}\,. \label{hotel}
\end{equation}
Putting \eqref{ilsolito} in \eqref{lofa} using \eqref{hotel} we get
\begin{equation}
s=\sum_{(x,y)\in
E'_N}\rho'([C_{\left\{x,y\right\}}])r^{[C_{\left\{x,y\right\}}]}
+\sum_{f\in F_N}[(\psi+c)(f)]_+r^{[C_{f}]}-r^\phi\,,
\label{ariacond2}
\end{equation}
where
$$
\rho'([C_{\left\{x,y\right\}}]:=\rho([C_{\left\{x,y\right\}}])+\sum_{\left\{f\in
F_N\,:\,(x,y)\in f\right\}}m(f)\,.
$$
Since the positive numbers $m(f)$ and $\rho([C_{x,y}])$ are
arbitrary, for any fixed $c\in \mathbb R$ \eqref{ariacond2} says
that $s\geq s(c)$ where
$$
s(c):= \sum_{f\in F_N}[(\psi+c)(f)]_+r^{[C_{f}]}-r^\phi\,.
$$

Recall that given $(x,y)\in E'_N$ there exist only two elements of
$F_N$ that contain $(x,y)$. One is clockwise oriented and the other
one is anticlockwise oriented. As before we call $f_+$ the element
of $F'_N$ such that $(x,y)\in f_+$ and $f_-$ the element of $F'_N$
such that $(x,y)\in f_-^c$.  If $c\in \mathbb R$ is fixed, the
condition $s\geq s(c)$ is equivalent to require for any $(x,y)\in
E'_N$ (the condition for $(y,x)$ will automatically be satisfied)
\begin{equation}
s(x,y)\geq [\psi(f_+)+c]_++[-\psi(f_-)-c]_+-[\phi(x,y)]_+\,.
\label{spagna}
\end{equation}
Since $\phi=d\psi $ we have also
\begin{equation}
\phi(x,y)=\psi(f_+)-\psi(f_-)\,. \label{equesta}
\end{equation}
It holds the following identity
\begin{equation}
Z(a,b):=[a]_++[-b]_+-[a-b]_+=\left\{
\begin{array}{ll}
\min\left\{|a|,|b|\right\} & \hbox{if} \ \ \hbox{sg}(a)=\hbox{sg}(b)\,,\\
0 & \hbox{if} \ \ \hbox{sg}(a)\neq \hbox{sg}(b)\,,
\end{array}
\right.
\nonumber
\end{equation}
where $a$ and $b$ are real numbers, $\hbox{sg}(\cdot)$ is the sign
function that associate to any real number its sign and the first
identity is the definition of the function $Z$. Putting
\eqref{equesta} in \eqref{spagna} we get for any $(x,y)\in E'_N$
\begin{equation}
s(x,y)\geq Z(\psi(f_+)+c,\psi(f_-)+c)\,.
\nonumber
\end{equation}
It is simple to check that
\begin{equation}
 Z(\psi(f_+)+c,\psi(f_-)+c)=d(0,c+I(x,y))\,,
\nonumber
\end{equation}
where $d$ is the Euclidean distance on the real line and $I(x,y)$ is
the closed interval of the real line
$[\min\left\{\psi(f_+),\psi(f_-)\right\},
\max\left\{\psi(f_+),\psi(f_-)\right\}]:=[i_1(x,y),i_2(x,y)]$.

Let us call
$$
S(c):=\left\{s\in R(0)\,:\, s\geq s(c)\right\}\,.
$$
We showed that
\begin{equation}
S(c)=\left\{s\in R(0)\,:\, s(x,y)\geq d(0,c+I(x,y))\,, (x,y)\in
E'_N\right\}\,. \nonumber
\end{equation}
Since $c$ is arbitrary we get that when $r\in \mathbb R^e\cap
R(\phi)$ the corresponding symmetric part $s$ in \eqref{urp}
satisfies
\begin{equation}
s\in \cup_{c\in \mathbb R}S(c). \label{fantacond}
\end{equation}
We now write this condition in a simpler form. Indeed we will now
show that $s\in R(0)$ belongs to $\cup_{c\in \mathbb R}S(c)$ if and
only if
\begin{equation}
\cap_{(x,y)\in E'_N}[-i_2(x,y)-s(x,y),-i_1(x,y)+s(x,y)]\neq
\emptyset\,.\label{finocamilla}
\end{equation}
Indeed it holds
\begin{equation}
s(x,y)\geq d(0,c+I(x,y)) \qquad \Longleftrightarrow \qquad c\in
\left[-i_2(x,y)-s(x,y),-i_1(x,y)+s(x,y)\right]\,. \label{equiequi}
\end{equation}
If $s\in \cup_cS(c)$ then, using \eqref{equiequi}, there exists a
$c^*\in \mathbb R$ such that
$$
c^*\in \left[-i_2(x,y)-s(x,y),-i_1(x,y)+s(x,y)\right]\,, \qquad
\forall (x,y)\in E'_N\,,
$$
and \eqref{finocamilla} holds. Conversely assume that
\eqref{finocamilla} holds and call $c^*$ an element of the non empty
intersection. Then still using \eqref{equiequi} we get $s\in S(c^*)$.

Since closed intervals are compact convex sets of $\mathbb R^1$, we
can write condition \eqref{finocamilla} in a simpler form using the
Helly's Theorem \ref{Helly}. We apply it in the special case of
dimension one. We obtain that condition \eqref{finocamilla} holds if
and only if for any pair $(x,y)$ and $(x',y')$ of elements of $E'_N$
it holds
\begin{equation}
[-i_2(x,y)-s(x,y),-i_1(x,y)+s(x,y)]\cap
[-i_2(x',y')-s(x',y'),-i_1(x',y')+s(x',y')]\neq \emptyset\,. \nonumber
\end{equation}
This is equivalent to require
\begin{equation}
d\left(I(x,y),I(x',y')\right)\leq s(x,y)+s(x',y')\,. \label{quasifine}
\end{equation}
Since \eqref{quasifine} has to be satisfied for any pair of elements
of $E'_N$ we deduce that condition \eqref{fantacond} is equivalent
to
\begin{equation}
s\in \mathcal P\left(\left\{I(x,y)\right\}_{(x,y)\in E'_N}\right)\,.
\nonumber
\end{equation}
\qed

\medskip

Let us briefly discuss how the above results can be used on a
problem concerning an infinite graph. Consider the infinite graph
having vertices $\mathbb Z^2$ and edges coinciding with the pairs of
vertices $(x,y)$ such that $d(x,y)=1$. Consider also a periodic
weight $r$, i.e. a weight such that there exist $N_1$ and $N_2$
integer numbers such that
$$
r(x,y)=r(x+(N_1,N_2),y+(N_1,N_2))\,, \qquad \forall (x,y)\,.
$$
Then the problem of finding a periodic cyclic decomposition of $r$,
i.e. a decomposition such that
$\rho([C])=\rho([\tau_{(N_1,N_2)}C])$, is strictly related to the
problem of finding a decomposition like \eqref{dec*} or
\eqref{decel} for a finite torus with $N_1\times N_2$ vertices.
Indeed in the case of elementary decompositions \eqref{decel} the
two problems are equivalent. Apart from an irrelevant dilation, an
elementary cycle on the discrete torus can be naturally interpreted
as an elementary cycle on the infinite graph. It is easy to see that
if \eqref{decel} holds on the discrete torus, then
\begin{equation}
\sum_{[C]\in \mathcal C^e}\sum_{x\in \mathbb
Z^d}\rho([C])r^{[\tau_xC]}\,, \label{devoandare}
\end{equation}
is a periodic cyclic decomposition on the infinite graph with
elementary cycles. In \eqref{devoandare} $\mathcal C^e$ are the
equivalence classes of elementary cycles on the discrete torus.
Conversely every elementary cycle on the infinite graph can be
projected onto an elementary cycle of the torus. By periodicity of
the weight $r$ you get a decomposition like \eqref{decel}. In the
case of decompositions \eqref{dec*} the relation among the two
problems is more subtle and we will not discuss here.

\subsection{Other topologies}
The proofs of Theorems \ref{bart} and \ref{lisa} are very similar to
the one of \ref{ugu}. We give only a sketch.

\medskip

\noindent {\it Proof of Theorem \ref{bart}} The proof of this
Theorem follows closely the lines of reasoning of the proof of
Theorem \ref{ugu}. In particular the two basic ingredients are the
following. The first one is that every edge $(x,y)\in E_N$ belongs
to only two elements of $F_N$. The second one is that we can choose
orientations in $F_N'$ in such a way that every pair of elements on
$F_N'$ is oriented in agreement. As a consequence given $\phi \in
d\Lambda^2$ the two chain $\psi$ such that $\phi=d\psi$ is defined
only up an additive constant. \qed

\medskip

\noindent {\it Proof of Theorem \ref{lisa}} The proof of this
Theorem uses the same type of arguments of Theorem \ref{ugu}. The
difference is that given $\phi\in d\Lambda^2$ then the two chain
$\psi$ such that $\phi=d\psi$ is uniquely determined. This is due to
the fact that the surface is not orientable. For any $(x,y)\in E_N'$
such that there exist $f_+,f_-\in F_N'$ such that $(x,y)\in f_+$ and
$(y,x)\in f_-$ the we get the constraint $s(x,y)\geq d(0,I(x,y))$.
For an $(x,y)\in E_N'$ for which for example there exist $f_1,f_2\in
F_N'$ such that $(x,y)\in f_i$, $i=1,2$, instead of \eqref{spagna}
and \eqref{equesta} we get
$$
s(x,y)\geq [\psi(f_1)]_++[\psi(f_2)]_+-[\psi(f_1)+\psi(f_2)]_+\,.
$$
Since for any pair of real numbers $a,b$ it holds
\begin{equation}
[a]_++[b]_+-[a+b]_+=\left\{
\begin{array}{ll}
\min\left\{|a|,|b|\right\} & \hbox{if} \ \ \hbox{sg}(a)\neq\hbox{sg}(b)\,,\\
0 & \hbox{if} \ \ \hbox{sg}(a)= \hbox{sg}(b)\,,
\end{array}
\right.
\nonumber
\end{equation}
we deduce in this case $s(x,y)\geq d(0,J(x,y))$. The Theorem
follows. \qed

\section{Applications}
\label{applications}

We start this section discussing a very general example. First we show how to
discretize in a natural way a smooth divergence free vector field by
a divergence free discrete vector field. Then we apply the results
of Theorem \ref{ugu}. In dimension $d=2$, the continuous version of
the Hodge decomposition is the following. Let $u\in \mathbb T^2$ and
$F=(F_1(u),F_2(u))$ be a smooth vector field on $\mathbb T^2$. The
Hodge decomposition says that there exist two smooth functions $f$
and $\psi$ such that
\begin{equation}
F=\nabla f+\nabla^{\bot}\psi +a_1(1,0)+a_2(0,1)\,.
\label{dechodgecont}
\end{equation}
In the above formula $\nabla$ is the gradient,
$\nabla^{\bot}\psi(u):=(\psi_{u_2}(u),-\psi_{u_1}(u))$ is the
orthogonal gradient and $a_i=\int_{\mathbb T^2}F_i(v)dv$. The
definition of the orthogonal gradient can be seen as a continuum
version of formula \eqref{inpsd} in Remark \ref{tusolotu}.
Decomposition \eqref{dechodgecont} says that the continuous version
of the discrete vector fields on $d\Lambda^2$ are the vector fields
on $\mathbb T^2$ of the form
\begin{equation}
F=(\psi_{u_2},-\psi_{u_1})\,, \label{dcont}
\end{equation}
for a suitable smooth function $\psi$. Clearly a vector field of the
type \eqref{dcont} has zero divergence since
\begin{equation}
\nabla\cdot \nabla^\bot\psi=\psi_{u_2 u_1}-\psi_{u_1 u_2}=0\,.
\label{esplosione}
\end{equation}
To any smooth vector field of the type \eqref{dcont} we can
associate a discrete vector field $\phi_N\in d\Lambda^2$ on the
discrete torus defined as
\begin{eqnarray}
& &
\phi_N(x,x+e^{(1)}/N):=\int_{x_2-\frac{1}{2N}}^{x_2+\frac{1}{2N}}\psi_{u_2}\left(x_1+\frac{1}{2N},y\right)dy
\nonumber\\
& &
=\psi\left(x_1+\frac{1}{2N},x_2+\frac{1}{2N}\right)-\psi\left(x_1+\frac{1}{2N},x_2-\frac{1}{2N}\right)\,,
\label{gu}
\end{eqnarray}
and
\begin{eqnarray}
& &
\phi_N(x,x+e^{(2)}/N):=-\int_{x_1-\frac{1}{2N}}^{x_1+\frac{1}{2N}}\psi_{u_1}\left(y,x_2+\frac{1}{2N}\right)dy
\nonumber\\
& &
=\psi\left(x_1-\frac{1}{2N},x_2+\frac{1}{2N}\right)-\psi\left(x_1+\frac{1}{2N},x_2+\frac{1}{2N}\right)\,.
\label{ga}
\end{eqnarray}
In \eqref{gu} and \eqref{ga} we fixed the values of the discrete
vector field coinciding with the fluxes of the continuous vector
field $\nabla^{\perp}\psi$ respectively across the segments
$$
c\left(x+e^{(1)}/(2N)-e^{(2)}/(2N)\right)+(1-c)\left(x+e^{(1)}/(2N)+e^{(2)}/(2N)\right)\,,
\qquad c\in [0,1]\,,
$$
and
$$
c\left(x-e^{(1)}/(2N)+e^{(2)}/(2N)\right)+(1-c)\left(x+e^{(1)}/(2N)+e^{(2)}/(2N)\right)\,,\qquad
c\in[0,1]\,,
$$
with associated normal vectors respectively $e^{(1)}$ and $e^{(2)}$.
According to these definitions it is easy to find a two chain
$\psi_N\in \Lambda^2$ such that $\phi_N=d\psi_N$. Given $f\in F'_N$
of the form $f=(x,x+e^{(1)}/N,x+e^{(1)}/N+e^{(2)}/N,x+e^{(2)}/N,x)$
we have
\begin{equation}
\psi_N(f):=\psi\left(x+e^{(1)}/(2N)+e^{(2)}/(2N)\right)\,. \nonumber
\end{equation}
The condition of zero divergence for $\phi_N$ can be also directly
derived by \eqref{esplosione}.

\smallskip

We us call
$$
M:=\max_{u,v\in \mathbb T^2}|\psi(u)-\psi(v)|\,.
$$
Fix a weight $r:=r^{\phi_N}+s\in R(\phi_N)$ with $s(x,y)\geq
\frac{M}{2}$ for any $(x,y)\in E_N$. Since $d(I(x,y),I(x',y'))\leq
M$ for any pair of elements of $E'_N$ then
$$
s(x,y)+s(x',y')\geq M\geq d(I(x,y),I(x',y'))\,,
$$
so that $r\in R^e$.

\smallskip

As an application of this example we discuss a direct consequence
for a random walk in a random environment. Consider the infinite
graph with vertices $\mathbb Z^2$ and oriented edges coinciding with
pairs of nearest neighbor vertices. Let $\psi$ be a smooth periodic
function on $\mathbb R^2$ with integer periods $N_1$ and $N_2$, i.e.
such that $\psi(u)=\psi(u+(N_1,N_2))$ for any $u\in \mathbb R^2$.
Let $U=(U_1,U_2)$ be a random variable uniform on
$[0,N_1]\times[0,N_2]\subseteq \mathbb R^2$. We construct the random
discrete vector field $\phi(\omega)$ defined by
\begin{align}
\phi(x,x+e^{(1)}; \omega):=
\tau_{U(\omega)}\left[\psi\left(x_1+\frac{1}{2},x_2+
\frac{1}{2}\right)-\psi\left(x_1+\frac{1}{2},x_2-\frac{1}{2}\right)\right]\,,\nonumber\\
\phi(x,x+e^{(2)};\omega):=\tau_{U(\omega)}\left[
\psi\left(x_1-\frac{1}{2},x_2+\frac{1}{2}\right)-\psi\left(x_1+\frac{1}{2},x_2+\frac{1}{2}\right)\right]\,.
\nonumber
\end{align}
Clearly $\phi(\omega)\in d\Lambda^2$ for any $\omega$ and moreover
the law of $\phi$ is invariant under translations. Let also
$\left\{U_{\left\{x,y\right\}}\right\}_{(x,y)\in E_N'}$ be a
collection of i.i.d. random variables taking values on an interval
$[a,b]$, $a,b> 0$ and independent from $U$. Here $E_N'$ are oriented edges
of a $N_1\times N_2$ discrete torus of mesh size $N=1$. We construct the random
weights
\begin{equation}
r(x,y;
\omega):=Z(x,\omega)^{-1}\left(r^{\phi(\omega)}(x,y)+U_{\left\{x,y\right\}}\right)\,,
\label{environment}
\end{equation}
where $Z(x,\omega)$ is the normalization constant that guarantees
the validity of \eqref{norm}. Note that
$U_{\left\{x,y\right\}}=U_{\left\{x+N_1,y+N_2\right\}}$ so that
$r(\omega)$ is a periodic enironment. Let
$$
M:=\max_{u,v\in [0,N_1]\times [0,N_2]}|\psi(u)-\psi(v)|\,.
$$
If $a\geq \frac M 2$ then the hypotheses in \cite{D} are satisfied
and the discrete time random walk in random environment having
transition probabilities determined by \eqref{environment} satisfies
a quenched Central Limit Theorem since it admits a periodic cyclic
decomposition into elementary cycles. Since the random environment
is periodic the result obtained is rather simple but the argument
can be generalized to the non periodic case \cite{DG}.

\smallskip
We end the paper with some comments. When $r\in W$ is given, to
determine if $r\in R^e$ we need to compute the associated discrete
vector field $\phi$ and then a two chain $\psi$ such that
$\phi=d\psi$. Once $\psi$ is computed we determine the corresponding
polyhedron $\mathcal P$ and finally we can check if the symmetric
part $s$ of $r$ belongs or not to $\mathcal P$. As discussed in
Remark \ref{tusolotu} the determination of $\psi$ requires a non
local computation. It is then useful to have some simple sufficient
conditions. For any pair $(x,y)$ and $(x',y')$ of elements of $E'_N$
we have
\begin{equation}
d(I(x,y),I(x',y'))\leq \max_{f,f'\in F'_N}|\psi(f)-\psi(f')|\,.
\label{sisveglia}
\end{equation}
If we are able to estimate from above the right hand side of
\eqref{sisveglia} by a constant $M$ and $s(x,y)\geq M/2$ for any
$(x,y)\in E_N'$ then $r\in R^e$. Consider $\tilde \psi$ the element
of $\tilde \Lambda_0$ associated to $\psi$ as in Remark
\ref{tusolotu} and use \eqref{iosoloio}. Consider an unoriented
graph having as vertices $\tilde V_N$ and as unoriented edges
$\left\{x,y\right\}$ with $(x,y)\in \tilde E_N'$. To every
unoriented edge we associate the weight $w(x,y):=|\phi^r(D(x,y))|$.
The weight of a path is the sum of the weights of its edges. These
weights introduce a metric structure on the graph. The distance
among two vertices is the minimal weight among all paths connecting
the two vertices. An upper bound to the right hand side of
\eqref{sisveglia} is then the diameter of this unoriented graph. It
can be estimated for example as the weight of any spanning tree.

\bigskip

\noindent {\bf  Acknowledgements}.  D.G. acknowledges the financial
support of PRIN 2009 protocollo n.2009TA2595.02  and thanks the
Department of Physics of the University ``La Sapienza" for the kind
hospitality. We thank the anonymous referee for a very
careful reading of the paper.

\end{document}